%% file: LALMforNLP.tex
\documentclass[11pt]{article}

\usepackage{epsfig,epsf,fancybox}
\usepackage{amsmath}
\usepackage{mathrsfs}
\usepackage{amssymb}
\usepackage{amsfonts}
\usepackage{graphicx}
\usepackage{color}
\usepackage[linktocpage,pagebackref,colorlinks,linkcolor=blue,anchorcolor=blue,citecolor=blue,urlcolor=blue,hypertexnames=false]{hyperref}
\usepackage{boxedminipage}
\usepackage{stmaryrd}
\usepackage{multirow}
\usepackage{booktabs}
\usepackage{accents}
\usepackage{cite}
\usepackage{float}
\usepackage[ruled,vlined,linesnumbered]{algorithm2e}

\input{macros.tex}
\begin{document}

\title{First-order methods for constrained convex programming based on linearized augmented Lagrangian function\thanks{This work is partly supported by NSF grant DMS-1719549.}}

\author{Yangyang Xu\thanks{\url{xuy21@rpi.edu}. Department of Mathematical Sciences, Rensselaer Polytechnic Institute, Troy, New York.}
}

\date{}

\maketitle

\begin{abstract}
First-order methods have been popularly used for solving large-scale problems. However, many existing works only consider unconstrained problems or those with simple constraint. In this paper, we develop two first-order methods for constrained convex programs, for which the constraint set is represented by affine equations and smooth nonlinear inequalities. Both methods are based on the classic augmented Lagrangian function. They update the multipliers in the same way as the augmented Lagrangian method (ALM) but employ different primal variable updates. The first method, at each iteration, performs a single proximal gradient step to the primal variable, and the second method is a block update version of the first one. 

For the first method, we establish its global iterate convergence as well as global sublinear and local linear convergence, and for the second method, we show a global sublinear convergence result in expectation. Numerical experiments are carried out on the basis pursuit denoising and a convex quadratically constrained quadratic program to show the empirical performance of the proposed methods. Their numerical behaviors closely match the established theoretical results.

\vspace{0.3cm}

\vspace{0.3cm}


\end{abstract}

\section{Introduction}
Recent years have witnessed the surge of first-order methods partly due to the increasingly big data involved in modern applications. Compared to second or higher-order methods, first-order ones only require gradient information and generally have much lower per-iteration complexity. However, many existing works on first-order methods are about problems without constraint or with easy-to-project constraint and/or with affine constraint.  

In this paper, we consider the generally constrained convex programming
\begin{equation}\label{eq:ccp}
\min_{x} f_0(x)\equiv g(x)+h(x), \st Ax=b,\ f_j(x)\le 0, j=1,\ldots, m,
\end{equation}
where $g$ and $f_j$ for $j=1,\ldots,m$ are convex and Lipschitz differentiable functions, and $h$ is a proper closed convex (possibly nondifferentiable) function. For practical efficiency of our algorithms, we will assume $h$ to be simple in the sense that its proximal mapping is easy to compute. However, our convergence results do not require this assumption.

Applications that can be formulated into \eqref{eq:ccp} appear in many areas including operations research, statistics, machine learning, engineering, just to name a few. Towards finding a solution to \eqref{eq:ccp}, we design algorithms that only need zeroth and first-order information of $g$ and $f_j, j=1,\ldots, m$, and the proximal mapping of $h$.

\subsection{Augmented Lagrangian method}
Our algorithms are based on augmented Lagrangian function of \eqref{eq:ccp}. In the literature, there are several different augmented Lagrangian functions (see \cite{ben1997penalty-barrier} for example), and we use the classic one. Let
$$\psi_\beta(u,v) = \left\{\begin{array}{ll}uv+\frac{\beta}{2}u^2,\,&\text{ if }\beta u +v \ge0,\\[0.2cm]
-\frac{v^2}{2\beta},\,&\text{ if }\beta u +v < 0,\end{array}\right.$$
and
$$\Psi_\beta(x,z)=\sum_{j=1}^m\psi_\beta(f_j(x),z_j).$$
Then the classic augmented Lagrangian function of \eqref{eq:ccp} is
\begin{equation}\label{eq:aug-fun}
\cL_\beta(x,y,z) = g(x)+h(x)+ y^\top(Ax-b) + \frac{\beta}{2}\|Ax-b\|^2 + \Psi_\beta(x,z),
\end{equation}
where $y$ and $z$ are Lagrangian multipliers, and $\beta>0$ is the penalty parameter.

The augmented Lagrangian method (ALM) for \eqref{eq:ccp}, at each iteration, renews $x$-variable by minimizing $\cL_\beta$ with respect to $x$ while $y$ and $z$ are fixed and then perform an augmented dual gradient ascent update to the multipliers $y$ and $z$, namely,
\begin{subequations}
\begin{align}
&x^{k+1}\in\argmin_x \cL_\beta(x, y^k,z^k),\label{eq:alm-x}\\
&y^{k+1}=y^k+\rho_y \nabla_y \cL_\beta(x^{k+1},y^k,z^k),\\
&z^{k+1}=z^k+\rho_z \nabla_z \cL_\beta(x^{k+1},y^k,z^k).
\end{align}
\end{subequations}
In general, it is difficult to solve the $x$-subproblem exactly or to a high accuracy. In one recent work \cite{xu2017ialm}, we show that if \eqref{eq:alm-x} is solved to a certain error tolerance, a global sublinear convergence of the inexact ALM can be established. In this work, we propose to perform one single proximal gradient update to \eqref{eq:alm-x}, and a sublinear convergence can still be shown.

\subsection{Related work}
ALM has been popularly used to solve constrained optimization problems; see books \cite{bertsekas1999nonlinear, bertsekas2014constrained}. However, most works on first-order methods in the ALM framework consider affinely constrained problems, and only a few study the methods for generally constrained problems in the form of \eqref{eq:ccp}. We review these works below.

For smooth affinely constrained convex programs, \cite{lan2016iteration-alm} analyzes the iteration complexity of an inexact ALM, where each primal subproblem is approximately solved by Nesterov's optimal first-order method \cite{nesterov2013gradient}. It shows that to reach an $\vareps$-optimal solution (see Definition \ref{def:eps-opt} below), $O(\vareps^{-\frac{7}{4}})$ gradient evaluations are sufficient. In addition, it shows that $O(\vareps^{-1}|\log\vareps|)$ gradient evaluations can guarantee an $\vareps$-optimal solution by an inexact proximal ALM. Although the number of gradient evaluations is not explicitly given, \cite{nedelcu2014computational, liu2016iALM} also consider inexact ALM. They specify the accuracy that each primal subproblem need be solved to and estimate the outer iteration complexity of the inexact ALM. Within the ALM framework, \cite{xu2017accelerated-alm} perform a single proximal gradient update to primal variable at each iteration and establish $O(\vareps^{-1})$ complexity result to have an $\vareps$-optimal solution for affinely constrained composite convex programs. This linearized ALM also appears as a special case of the methods in \cite{hong2014block, cui2016convergence, GXZ-RPDCU2016, xu2017accelerated-pdbc, gao2017first}, which perform Gauss-Seidel or randomized block coordinate update to the primal variable in the ALM framework. 

Towards finding solutions of general saddle-point problems, \cite{nedic2009subgradient} gives a subgradient method. If both primal and dual constraint sets are compact, the method has $O(1/\sqrt{k})$ convergence rate in terms of primal-dual gap, where $k$ is the number of iterations. It also discusses how to apply the subgradient method to convex optimization problems with nonlinear inequality constraint.  On smooth constrained convex problems, \cite{yu2016primal} proposes a primal-dual type first-order method (see \eqref{eq:pd-yn} in section \ref{sec:qcqp}). Assuming compactness of the constraint set, it establishes $O(\vareps^{-1})$ iteration complexity result to produce an $\vareps$-optimal solution. Recently, \cite{xu2017ialm} studies an inexact ALM for \eqref{eq:ccp} and proposes to use Nesterov's optimal first-order method to approximately solve each $x$-subproblem. When the constraint set is bounded, it shows that nearly $O(\vareps^{-\frac{3}{2}})$ gradient evaluations suffice to obtain an $\vareps$-optimal solution, and for the smooth case, the result can be improved to $O(\vareps^{-1}|\log\vareps|)$. Compared to these works, our iteration complexity results will be better under weaker assumptions. 

\subsection{Contributions} This paper mainly makes the following contributions.
\begin{itemize}
\item We propose a first-order method, named LALM, for solving composite convex problems with both affine equality and smooth nonlinear inequality constraints. The method is based on proximal linearization of the classic augmented Lagrangian function. Under mild assumptions, we show global iterate sequence convergence of LALM to a primal-dual optimal solution.

\item Also, we analyze the iteration complexity of the proposed method. We show that to reach an $\vareps$-optimal solution, $O(\vareps^{-1})$ gradient evaluations are sufficient. In addition, we establish its local linear convergence by assuming the existence of a non-degenerate primal-dual solution and positive definiteness of Hessian of the augmented Lagrangian function near the non-degenerate primal-dual solution.  

\item Furthermore, as the problem has the so-called coordinate friendly structure, we propose a block update version of LALM. At each iteration, the method renews a single block coordinate while keeping all the other coordinates unchanged and then immediately performs an update to dual variables. We show that in expectation, an $\vareps$-optimal solution can be obtained by $O(\vareps^{-1})$ gradient evaluations.

\item We implement LALM and its block update version and apply them to the basis pursuit denoising problem and a quadratically constrained quadratic program. On both problems, we notice better performance of the block-LALM in terms of iteration number. In addition, when the iterate is far away from optimality, sublinear convergence is observed, and while the iterate approaches to optimality, both methods converge linearly.
\end{itemize}

\subsection{Notation and organization}
We focus on finite-dimensional Euclidean space, but our analysis can be directly extended to a Hilbert space. We use $[m]$ as the set $\{1,2,\ldots,m\}$, and $[a]_+=\max(0,a)$ denotes the positive part of a real number $a$. We use $I$ as the identity matrix. Given a symmetric positive definite (SPD) matrix $P$, we define $\|x\|_P=\sqrt{x^\top P x}$, and if $P=I$, we simply write it as $\|x\|$. Also, given a nonnegative vector $\vell=[\ell_1,\ldots,\ell_n]\in\RR^n$, we define $\|x\|_\vell^2=\sum_{i=1}^n \ell_i \|x_i\|^2$ if $x$ is partitioned into $n$ blocks $(x_1,\ldots,x_n)$. For any convex function $f(x)$, we use $\tilde{\nabla}f(x)$ as its subgradient and $\partial f(x)$ the subdifferential of $f$ at $x$, i.e., the set of all subgradients at $x$. When $f$ is differentiable, $\tilde{\nabla} f(x)$ coincides with the gradient of $f$, and we simply write it to $\nabla f(x)$. The indicator function of a set $\cX$ is defined as $\iota_\cX(x)=0$ if $x\in \cX$ and $+\infty$ otherwise. $B_\gamma(x)$ represents a ball with radius $\gamma$ and center $x$. $\EE_{i_k}$ denotes the expectation about $i_k$ conditioned on all previous history.

For ease of notation, we use $w$ as the triple $(x,y,z)$ and denote the smooth part of $\cL_\beta$ as 
$$F_\beta(w) = \cL_\beta(w) - h(x).$$
In addition, we define
\begin{equation}\label{eq:def-Phi}
\Phi(\bar{x}; x,y,z)=f_0(\bar{x})-f_0(x)+y^\top (A\bar{x}-b) + \sum_{j=1}^m z_j f_j(\bar{x}).
\end{equation}

\begin{definition}[$\vareps$-optimal solution]\label{def:eps-opt}
Let $f_0^*$ be the optimal value of \eqref{eq:ccp}. We call $\bar{x}$ an $\vareps$-optimal solution to \eqref{eq:ccp} if
$$|f_0(\bar{x})-f_0^*| \le \vareps,\quad \|A\bar{x}-b\| + \sum_{j=1}^m [f_j(\bar{x})]_+\le \vareps.$$
\end{definition}

\textbf{Organization.} The rest of the paper is organized as follows. Section \ref{sec:pre} gives several technical results that will be used to prove our main theorems. We propose a linearized ALM for \eqref{eq:ccp} in section \ref{sec:lalm} and a block linearized ALM in section \ref{sec:blalm}. Convergence results are also given. In section \ref{sec:application}, we discuss a few applications and how the proposed methods can be applied. Numerical results are given in section \ref{sec:numerical}, and finally section \ref{sec:conclusion} concludes the paper. 

\section{Technical assumptions and preliminary results}\label{sec:pre}
A point $w=(x,y,z)$ satisfies the Karush-Kuhn-Tucker (KKT) conditions for \eqref{eq:ccp} if
\begin{subequations}\label{eq:kkt}
\begin{align}
0\in \nabla g(x) + \partial h(x) + A^\top y + \sum_{j=1}^m z_j \nabla f_j(x), \label{eq:kkt-1st}&\\
Ax = b, \label{eq:kkt-res}&\\
z_j\ge0, \, f_j(x)\le 0,\, z_jf_j(x)=0, \forall j\in [m]. \label{eq:kkt-cp}
\end{align}
\end{subequations}
If $w$ satisfies the above conditions, we call it a KKT point. For convex programs, the conditions in \eqref{eq:kkt} are sufficient for $x$ to be an optimal solution of \eqref{eq:ccp}. If a certain qualification condition (e.g., the Slater condition) holds, they are also necessary.

\subsection{Technical assumptions}
Throughout the paper, we assume the existence of a KKT point.
\begin{assumption}\label{assump-kkt}
There exists a point $w^*=(x^*,y^*,z^*)$ satisfying the KKT conditions in \eqref{eq:kkt}.
\end{assumption}
Under the above assumption, it follows from the convexity of $f_0$ that
\begin{equation}\label{eq:opt}
\Phi(x;w^*) \ge 0,\, \forall x,
\end{equation}
where $\Phi$ is defined in \eqref{eq:def-Phi}.

In addition, we make the following assumption, which holds if $\dom(h)$ is bounded.
\begin{assumption}\label{assump-lip}
There are constants $L_g,L_1,\ldots, L_m$ and $B_1,\ldots,B_m$ such that
\begin{align}
\|\nabla g(\hat{x})-\nabla g(\tilde{x})\|\le L_g\|\hat{x}-\tilde{x}\|,\,&\forall \hat{x}, \tilde{x} \in \dom(h),\label{eq:lip-g}\\
\|\nabla f_j(\hat{x})-\nabla f_j(\tilde{x})\|\le L_j\|\hat{x}-\tilde{x}\|,\,& \forall \hat{x}, \tilde{x} \in \dom(h), \forall j\in [m],\label{eq:lip-f}\\
\|\nabla f_j(x)\| \le B_j,\,& \forall x\in\dom(h), \forall j\in [m]. \label{eq:bd-fgrad}
\end{align} 
\end{assumption}

From the mid-point theorem, the boundedness of $\nabla f_j$ implies the Lipschitz continuity of $f_j$, i.e.,
\begin{equation}\label{eq:lip-fval}
|f_j(\hat{x})-f_j(\tilde{x})| \le B_j \|\hat{x}-\tilde{x}\|,\,\forall \hat{x}, \tilde{x} \in \dom(h), \forall j\in [m].
\end{equation}

\subsection{Preparatory lemmas}
In this subsection, we give several lemmas that will be used multiple times in our convergence analysis. First we show the Lipschitz continuity of $\nabla_x\Psi(w)$ with respect to $x$.
\begin{lemma}\label{lem:lip}
Under Assumption \ref{assump-lip}, we have
\begin{equation}\label{eq:lip-Psi}
\|\nabla_x \Psi(\hat{x}, z)-\nabla_x \Psi(x,z)\| \le L_\Psi(x,z) \|\hat{x}-x\|,\, \forall \hat{x}, x, z,
\end{equation}
where
\begin{equation}\label{eq:def-Lpsi-xz}
L_\Psi(x,z) = \sum_{j=1}^m \left(\beta B_j^2 + L_j\big[\beta f_j(x)+z_j\big]_+\right)
\end{equation}
\end{lemma}

\begin{proof}
First we notice that $\frac{\partial}{\partial u}\psi_\beta(u,v)=[\beta u+v]_+$, and thus for any $v$,
$$\left|\frac{\partial}{\partial u}\psi_\beta(\hat{u},v)-\frac{\partial}{\partial u}\psi_\beta(\tilde{u},v)\right|\le \beta|\hat{u}-\tilde{u}|,\,\forall \hat{u},\tilde{u}.$$
Let $h_j(x,z_j) = \psi_\beta(f_j(x),z_j),\,j=1,\ldots,m$. Then
\begin{align}\label{eq:der-lip-psi}
&\|\nabla_x h_j(\hat{x},z_j)-\nabla_x h_j(x,z_j)\|\nonumber\\
=&\big\|\frac{\partial}{\partial u}\psi_\beta(f_j(\hat{x}),z_j)\nabla f_j(\hat{x})-\frac{\partial}{\partial u}\psi_\beta(f_j(x),z_j)\nabla f_j(x)\big\|\nonumber\\
\le&\big\|\frac{\partial}{\partial u}\psi_\beta(f_j(\hat{x}),z_j)\nabla f_j(\hat{x})-\frac{\partial}{\partial u}\psi_\beta(f_j(x),z_j)\nabla f_j(\hat{x})\big\|\nonumber\\
&+\big\|\frac{\partial}{\partial u}\psi_\beta(f_j(x),z_j)\nabla f_j(\hat{x})-\frac{\partial}{\partial u}\psi_\beta(f_j(x),z_j)\nabla f_j(x)\big\|\nonumber\\
\le& \beta|f_j(\hat{x})-f_j(x)|\cdot\|\nabla f_j(\hat{x})\|+\big|\frac{\partial}{\partial u}\psi_\beta(f_j(x),z_j)\big|\cdot\|\nabla f_j(\hat{x})-\nabla f_j(x)\|\\
\le & \beta B_j^2\|\hat{x}-x\|+ L_j\big[\beta f_j(x)+z_j\big]_+ \cdot \|\hat{x}-x\|.\nonumber
\end{align}
Hence,
\begin{align*}
\|\nabla_x \Psi_\beta(\hat{x},z)-\nabla_x \Psi_\beta(x,z)\|
\le  \sum_{j=1}^m \|\nabla_x h_j(\hat{x},z_j)-\nabla_x h_j(x,z_j)\|
\le  L_\Psi(x,z)\|\hat{x}-x\|,
\end{align*}
which completes the proof.
\end{proof}

\begin{remark}
Note that the Lipschitz constant $L_\Psi(x,z)$ in \eqref{eq:lip-Psi} depends on the point $(x,z)$ and is not a universal constant. We will set its value at the iterate of the algorithm. Together with the next lemma, the inequality in \eqref{eq:lip-Psi} implies that a sufficient progress can be obtained after each $x$-update.
\end{remark}

\begin{lemma}\label{lem:lip-ineq}
For a continuously differentiable function $\phi(u)$ and a given $v$, if $\|\nabla \phi(u)-\nabla \phi(v)\| \le L_\phi(v)\|u-v\|,\,\forall u$, then
$$\phi(u) \le \phi(v) + \langle \nabla \phi(v), u-v \rangle + \frac{L_\phi(v)}{2}\|u-v\|^2.$$
\end{lemma}

The following result is easy to show (c.f., \cite[Prop. 2.3]{combettes2015stochastic}). It will be used for establishing iterate convergence of the proposed algorithm.
\begin{lemma}\label{lem:fejer}
Let $\{P^k\}$ be a sequence of SPD matrices, and there are SPD matrices $\underline{P}$ and $\overline{P}$ such that $\underline{P}\succeq P^k \succeq \overline{P}$. Let $\cW$ be a nonempty set. If the sequence $\{w^k\}$ satisfies
$$\|w^{k+1}-w\|_{P^{k+1}}^2 \le \|w^k-w\|_{P^k}^2,\,\forall w\in \cW,$$
and $\{w^k\}$ has a cluster point $\bar{w}$ in $\cW$, then $w^k$ converges to $\bar{w}$.
\end{lemma}

The result below will be used to establish convergence rate of our algorithms. It is similar to a deterministic result in \cite{xu2017ialm} and can be shown in the same way. We omit its proof. 
\begin{lemma}\label{lem:pre-rate}
Assume $(x^*,y^*,z^*)$ is a KKT point of \eqref{eq:ccp}. Let $\bar{x}$ be a stochastic point such that for any $y$ and any $z\ge0$,
\begin{equation}\label{eq:pre-rate-cond}
\EE [\Phi(\bar{x}; x^*, y,z)] \le \alpha + c_1\|y\|^2+c_2\|z\|^2,
\end{equation}
where $\alpha$ and $c_1, c_2$ are nonnegative constants independent of $y$ and $z$. Then
\begin{align}
-\left(\alpha + 4c_1\|y^*\|^2+ 4c_2\sum_{j=1}^m(z_j^*)^2\right)\le \EE [f_0(\bar{x})-f_0(x^*)]\le \alpha,\label{eq:pre-rate-obj}\\
\EE \|A\bar{x}-b\| + \sum_{j=1}^m\EE [f_j(\bar{x})]_+ \le \alpha + c_1\big(1+\|y^*\|\big)^2 + c_2\sum_{j=1}^m\big(1+z_j^*\big)^2. \label{eq:pre-rate-res}
\end{align}
\end{lemma}

\section{Linearized augmented Lagrangian method}\label{sec:lalm}
In this section, we propose a linearized augmented Lagrangian method (LALM). Different from the step in \eqref{eq:alm-x}, it updates $x$-variable by a single proximal gradient descent of the augmented Lagrangian function. The method is summarized in Algorithm \ref{alg:lalm}, where $\delta\ge 0$ is a constant and $$L_F^k=L_g+\beta\|A\|^2+L_\Psi(x^k,z^k)$$ with $L_\Psi$ defined in \eqref{eq:lip-Psi}. 

Note that the setting of $\eta^k$ is for simplicity of our analysis. Practically, one can choose it by starting from $\eta^{k-1}$ and then backtracking such that
\begin{equation}\label{eq:choose-eta}
F_\beta(x^{k+1},y^k, z^k) \le F_\beta(w^k) + \big\langle \nabla_x F_\beta(w^k), x^{k+1}-x^k\big\rangle + \frac{\eta^k}{2}\|x^{k+1}-x^k\|^2,
\end{equation}
and all our convergence results can still be shown.
When $\eta^k\ge L_F^k$, the above inequality always holds from Lemma \ref{lem:lip-ineq}. 

\begin{algorithm}
\caption{Linearized augmented Lagrangian method (LALM) for \eqref{eq:ccp}}\label{alg:lalm}
\DontPrintSemicolon
\textbf{Initialization:} choose $x^0, y^0, z^0$ and $\beta,\rho_y,\rho_z, \delta\ge0$; set $\eta^{-1}=0$\;
\For{$k=0,1,\ldots$}{
Let $\eta^k = \max(\eta^{k-1}, L_F^k+\delta)$\;
Perform the updates
\begin{subequations}
\begin{align}
x^{k+1}=&~\argmin_x h(x) + \big\langle \nabla_x F_\beta(w^k), x\big\rangle + \frac{\eta^k}{2}\|x-x^k\|^2,\label{eq:lalm-x}\\
y^{k+1}=&~y^k + \rho_y (Ax^{k+1}-b),\label{eq:lalm-y},\\
z^{k+1}_j=&~z_j^k+\rho_z \cdot\max\left(-\frac{z_j^k}{\beta}, f_j(x^{k+1})\right), j=1,\ldots, m.\label{eq:lalm-z}
\end{align}
\end{subequations}
}
\end{algorithm}

\subsection{Global convergence analysis}

To show the convergence results of Algorithm \ref{alg:lalm}, we need the following two lemmas, which can be found in \cite{xu2017ialm}.
\begin{lemma}\label{lem:multiplier}
Let $y$ and $z$ be updated by \eqref{eq:lalm-y} and \eqref{eq:lalm-z} respectively. Then for any $k$, it holds
\begin{align}
\frac{1}{2\rho_y}\big[\|y^{k+1}-y\|^2-\|y^k-y\|^2+\|y^{k+1}-y^k\|^2\big] - \langle y^{k+1}-y, r^{k+1}\rangle = &~ 0,\label{eq:lalm-yterm}\\
\frac{1}{2\rho_z}\big[\|z^{k+1}-z\|^2-\|z^k-z\|^2+\|z^{k+1}-z^k\|^2\big]- \sum_{j=1}^m (z_j^{k+1}-z_j)\cdot\max\big(-\frac{z_j^k}{\beta}, f_j(x^{k+1})\big)=& ~0,\label{eq:lalm-zterm}
\end{align}
where $r^k=Ax^k-b$.
\end{lemma}

\begin{lemma}\label{lem:ineq-z}
For any $z\ge0$, we have
\begin{align}\label{eq:lalm-ineq-z}
&\frac{\beta-2\rho_z}{2\rho_z^2}\|z^{k+1}-z^k\|^2\le \Psi_\beta(x^{k+1},z^k)-\sum_{j=1}^m z_j f_j(x^{k+1})- \sum_{j=1}^m (z_j^{k+1}-z_j)\cdot\max\big(-\frac{z_j^k}{\beta}, f_j(x^{k+1})\big).
\end{align}
\end{lemma}

Using the above two lemmas, we establish a fundamental result on Algorithm \ref{alg:lalm}.
\begin{theorem}[One-iteration progress of LALM]\label{thm:1iter}
Let $\{w^k\}$ be the sequence generated from Algorithm \ref{alg:lalm}. Then for any $x$ such that $Ax=b$ and $f_j(x)\le 0,\,\forall j\in [m]$, any $y$, and any $z\ge0$, it holds that
\begin{align}\label{eq:lalm-ineq2}
&\Phi(x^{k+1}; w)+\frac{\eta^k}{2}\|x^{k+1}-x\|^2 + \frac{1}{2\rho_y}\|y^{k+1}-y\|^2+ \frac{1}{2\rho_z}\|z^{k+1}-z\|^2\nonumber\\
&+\frac{\beta-\rho_y}{2}\|r^{k+1}\|^2+\frac{\beta-\rho_z}{2\rho_z^2}\|z^{k+1}-z^k\|^2+ \frac{1}{2}\|x^{k+1}-x^k\|^2_{(\eta^k-L_g-L_\Psi^k)I-\beta A^\top A}\nonumber\\
\le &\frac{\eta^k}{2}\|x^k-x\|^2 -\frac{\beta}{2}\|r^k\|^2+ \frac{1}{2\rho_y}\|y^k-y\|^2+ \frac{1}{2\rho_z}\|z^k-z\|^2,
\end{align}
where $\Phi$ is defined in \eqref{eq:def-Phi}, and $L_\Psi^k = L_\Psi(x^k, z^k)$ with $L_\Psi$ defined in \eqref{eq:def-Lpsi-xz}.
\end{theorem}

\begin{proof}
From the update in \eqref{eq:lalm-x}, it follows that
\begin{equation}\label{eq:opt-cond-k}
0 \in \partial h(x^{k+1}) + \nabla g(x^k) + A^\top y^k + \beta A^\top r^k + \nabla_x \Psi(x^k,z^k) + \eta^k (x^{k+1}-x^k).
\end{equation}
By the convexity of $h$, we have
\begin{equation}\label{eq:ineq-h}
\left\langle x^{k+1}-x, \tilde{\nabla} h(x^{k+1})\right\rangle \ge h(x^{k+1})-h(x),
\end{equation}
From the convexity of $g$ and $\Psi(\cdot,z)$, and also Lemmas \ref{lem:lip} and \ref{lem:lip-ineq}, we have
\begin{align}\label{eq:ineq-g-psi}
&~\left\langle x^{k+1}-x, \nabla g(x^k) + \nabla_x \Psi(x^k,z^k) \right\rangle\cr
=&~\left\langle x^{k+1}-x^k, \nabla g(x^k) + \nabla_x \Psi(x^k,z^k) \right\rangle + \left\langle x^{k}-x, \nabla g(x^k) + \nabla_x \Psi(x^k,z^k) \right\rangle\cr
 \ge & ~g(x^{k+1})+\Psi(x^{k+1},z^k) - g(x^k)-\Psi(x^k,z^k)-\frac{L_g+L_\Psi^k}{2}\|x^{k+1}-x^k\|^2\cr
 &~+g(x^k)+\Psi(x^k,z^k)-g(x)-\Psi(x,z^k)\cr
 =&~g(x^{k+1})+\Psi(x^{k+1},z^k) - g(x)-\Psi(x,z^k)-\frac{L_g+L_\Psi^k}{2}\|x^{k+1}-x^k\|^2.
\end{align}
For $x$ such that $Ax=b$, it holds that
\begin{align}\label{eq:ineq-y-r}
&~\langle x^{k+1}-x, A^\top y^k + \beta A^\top r^k \rangle\cr 
=&~ \langle r^{k+1}, y^{k+1}-\rho_y r^{k+1}+\beta r^k\rangle\cr
= &~y^\top r^{k+1} + \langle y^{k+1}-y, r^{k+1}\rangle +(\beta-\rho_y)\|r^{k+1}\|^2 + \beta \langle r^{k+1},r^k- r^{k+1}\rangle\cr
=&~y^\top r^{k+1} + \langle y^{k+1}-y, r^{k+1}\rangle +(\beta-\rho_y)\|r^{k+1}\|^2 - \frac{\beta}{2} \left[ \|r^{k+1}\|^2 - \|r^k\|^2 + \|r^{k+1}- r^{k}\|^2\right].
\end{align}
Adding \eqref{eq:ineq-h}, \eqref{eq:ineq-g-psi}, \eqref{eq:ineq-y-r}, and the following equation
$$\big\langle x^{k+1}-x, \eta^k(x^{k+1}-x^k)\big\rangle=\frac{\eta^k}{2}\left[\|x^{k+1}-x\|^2-\|x^k-x\|^2+\|x^{k+1}-x^k\|^2\right],$$
we have from \eqref{eq:opt-cond-k} that
\begin{align}\label{eq:lalm-ineq-a}
&f_0(x^{k+1})-f_0(x)+\Psi(x^{k+1},z^k)-\Psi(x,z^k)+y^\top r^{k+1}+ \langle y^{k+1}-y, r^{k+1}\rangle\cr
&  -\frac{\beta}{2} \left[\|r^{k+1}\|^2 + \|r^{k+1}- r^{k}\|^2\right] +(\beta-\rho_y)\|r^{k+1}\|^2+\frac{\eta^k}{2}\|x^{k+1}-x\|^2 + \frac{\eta^k-L_g-L_\Psi^k}{2}\|x^{k+1}-x^k\|^2\cr
\le &\frac{\eta^k}{2}\|x^k-x\|^2-\frac{\beta}{2}\|r^k\|^2.
\end{align}
The desired result is obtained by noting $\Psi(x,z^k) \le 0$, adding \eqref{eq:lalm-yterm}, \eqref{eq:lalm-zterm}, and \eqref{eq:lalm-ineq-z} to the above inequality, and rearranging terms.
\end{proof}

The next lemma shows the upper boundedness of $\eta^k$.
\begin{lemma}\label{lem:eta-up-bd}
Let $\{w^k\}$ be the sequence generated from Algorithm \ref{alg:lalm} with $z_j^0\ge0,\forall j\in [m]$. If $\rho_y, \rho_z \in (0,\beta]$, then $\eta^k \le \bar{\eta},\,\forall k\ge0$, where $\bar{\eta}$ is a constant satisfying
\begin{equation}\label{eq:bar-eta-cond}
\begin{aligned}
\bar{\eta} \ge &~\delta + L_g + \beta\|A\|^2 + \beta\sum_{j=1}^m B_j^2 + \beta\sum_{j=1}^m B_j L_j\left(\|x^0-x^*\|+\frac{\|y^0-y^*\|}{\sqrt{\rho_y\eta^0}}+\frac{\|z^0-z^*\|}{\sqrt{\rho_z\eta^0}}\right)\\
&+\sqrt{\sum_{j=1}^m L_j^2}\left(\sqrt{\rho_z\bar{\eta}}\|x^0-x^*\|+\sqrt{\frac{\rho_z\bar{\eta}}{\rho_y\eta^0}}\|y^0-y^*\|+\|z^*\|+\max\big(1,\sqrt{\frac{\bar{\eta}}{\eta^0}}\big)\|z^0-z^*\|\right).
\end{aligned}
\end{equation}
\end{lemma}

\begin{proof}
Since $z_j^0\ge0,\forall j\in[m]$, we have $z_j^k\ge 0, \forall j\in[m]$ from the update of $z$ and the condition $\rho_z\in (0,\beta]$. Note $f_j(x^k) \le f_j(x^*) + B_j \|x^k-x^*\|\le B_j \|x^k-x^*\|$. It follows from the increasing monotonicity of $[a]_+$ that
\begin{equation*}
\sum_{j=1}^m L_j [\beta f_j(x^k)+z_j^k]_+ \le \sum_{j=1}^m \big(\beta B_j L_j \|x^k-x^*\| + L_j z_j^k\big),
\end{equation*}
and thus
\begin{align}\label{eq:fzj}
L_\Psi^k \le &~ \beta\sum_{j=1}^m B_j^2 + \sum_{j=1}^m \big(\beta B_j L_j \|x^k-x^*\| + L_j z_j^k\big)\cr
\le &~ \beta\sum_{j=1}^m B_j^2 + \beta \sum_{j=1}^m B_j L_j \|x^k-x^*\| + \sqrt{\sum_{j=1}^m L_j^2} (\|z^*\|+\|z^k-z^*\|).
\end{align}

We next show the desired result by induction. First, the result for $k=0$ directly follows from \eqref{eq:bar-eta-cond} and \eqref{eq:fzj}. Assume $\eta^k\le \bar{\eta},\,\forall k \le K-1.$ Then letting $w=w^*$ in \eqref{eq:lalm-ineq2}, we have from \eqref{eq:opt} and by dropping nonnegative terms on the left hand side that
\begin{align}\label{eq:tele-ineq}
&\frac{\eta^k}{2}\|x^{k+1}-x^*\|^2+ \frac{1}{2\rho_y}\|y^{k+1}-y^*\|^2+ \frac{1}{2\rho_z}\|z^{k+1}-z^*\|^2\nonumber\\
\le &\frac{\eta^k}{2}\|x^k-x^*\|^2 + \frac{1}{2\rho_y}\|y^k-y^*\|^2+ \frac{1}{2\rho_z}\|z^k-z^*\|^2.
\end{align}
Since $\eta^{k+1}\ge \eta^k$, dividing by $\eta^k$ on both sides of the above inequality yields
\begin{align}\label{eq:tele-ineq2}
&\frac{1}{2}\|x^{k+1}-x^*\|^2+ \frac{1}{2\rho_y\eta^{k+1}}\|y^{k+1}-y^*\|^2+ \frac{1}{2\rho_z\eta^{k+1}}\|z^{k+1}-z^*\|^2\nonumber\\
\le &\frac{1}{2}\|x^k-x^*\|^2 + \frac{1}{2\rho_y\eta^k}\|y^k-y^*\|^2+ \frac{1}{2\rho_z\eta^k}\|z^k-z^*\|^2.
\end{align}

Repeatedly using \eqref{eq:tele-ineq2} and also from \eqref{eq:tele-ineq}, we have
\begin{align*}
&~\frac{1}{2}\|x^K-x^*\|^2 + \frac{1}{2\rho_y\eta^{K-1}}\|y^K-y^*\|^2+ \frac{1}{2\rho_z\eta^{K-1}}\|z^K-z^*\|^2\\
\le& ~\frac{1}{2}\|x^0-x^*\|^2 + \frac{1}{2\rho_y\eta^0}\|y^0-y^*\|^2+ \frac{1}{2\rho_z\eta^0}\|z^0-z^*\|^2.
\end{align*}
The above inequality together with $\eta^{K-1}\le \bar{\eta}$ implies 
$$\|x^K-x^*\|\le \|x^0-x^*\|+\frac{\|y^0-y^*\|}{\sqrt{\rho_y\eta^0}}+\frac{\|z^0-z^*\|}{\sqrt{\rho_z\eta^0}}$$
and 
$$\|z^K-z^*\|\le \sqrt{\rho_z\bar{\eta}}\|x^0-x^*\|+\sqrt{\frac{\rho_z\bar{\eta}}{\rho_y\eta^0}}\|y^0-y^*\|+\sqrt{\frac{\bar{\eta}}{\eta^0}}\|z^0-z^*\|.$$
Hence, $\bar{\eta} \ge L_F^K+\delta$ from \eqref{eq:bar-eta-cond}, \eqref{eq:fzj}, and the above two inequalities. This completes the proof.
\end{proof}

We are now ready to show our main convergence and rate results.
\begin{theorem}[Iterate convergence of LALM]\label{thm:convg-lalm}
Under Assumptions \ref{assump-kkt} and \ref{assump-lip}, let $\{w^k\}$ be the sequence generated from Algorithm \ref{alg:lalm} with any $x^0, y^0$, and $z^0_j\ge0,\,\forall j\in[m]$. If $\rho_y,\rho_z\in (0,\beta)$ and $\delta>0$, 
 then $w^k$ converges to a KKT point $\bar{w}=(\bar{x},\bar{y},\bar{z})$ of \eqref{eq:ccp}.
\end{theorem}
\begin{proof}
Letting $w=w^*$ in \eqref{eq:lalm-ineq2} and dividing by $\eta^k$, we have from \eqref{eq:opt} that
\begin{align}\label{eq:lalm-ineq3}
&\frac{1}{2}\|x^{k+1}-x^*\|^2 + \frac{1}{2\rho_y\eta^k}\|y^{k+1}-y^*\|^2+ \frac{1}{2\rho_z\eta^k}\|z^{k+1}-z^*\|^2\nonumber\\
&+\frac{\beta-\rho_y}{2\eta^k}\|r^{k+1}\|^2+\frac{1}{2\rho_z^2\eta^k}(\beta-\rho_z)\|z^{k+1}-z^k\|^2+ \frac{1}{2\eta^k}\|x^{k+1}-x^k\|^2_{(\eta^k-L_g-L_\Psi^k)I-\beta A^\top A}\nonumber\\
\le &\frac{1}{2}\|x^k-x^*\|^2 + \frac{1}{2\rho_y\eta^k}\|y^k-y^*\|^2+ \frac{1}{2\rho_z\eta^k}\|z^k-z^*\|^2.
\end{align}
Summing up \eqref{eq:lalm-ineq3} over $k$ and noting $\eta^{k+1}\ge \eta^k$, we have from the condition $\delta>0$, $\rho_y,\rho_z\in (0,\beta)$ and Lemma \ref{lem:eta-up-bd} that
\begin{equation}\label{eq:diff-lim}
\lim_{k\to\infty}x^{k+1}-x^k = 0, \quad \lim_{k\to\infty} y^{k+1}-y^k = \lim_{k\to\infty}\rho_y r^{k+1}=0,\quad \lim_{k\to\infty} z^{k+1}-z^k=0.
\end{equation}
In addition, it follows from \eqref{eq:lalm-ineq3} that $\{w^k\}$ is bounded and must have a cluster point $\bar{w}$. Hence, there is a subsequence $\{w^k\}_{k\in \cK}$ convergent to $\bar{w}$. Since $\{\eta^k\}$ is increasing and bounded, it must converge to a number $\eta^\infty$. 

Below we show that $\bar{w}$ is a KKT point. First, we have  $A\bar{x}-b=0$ from \eqref{eq:diff-lim}, i.e, $\bar{x}$ satisfies \eqref{eq:kkt-res}. 

Secondly, from the update of $z$ and $\rho_z < \beta$, it follows $z_j^k\ge0,\,\forall j\in [m], \forall k$, and thus $\bar{z}_j\ge0,\forall j\in [m]$. If $f_j(x^{k+1}) >0$, then $f_j(x^{k+1})=\frac{1}{\rho_z}(z_j^{k+1}-z_j^k) \to 0$ that indicates $[f_j(x^{k+1})]_+\to 0$. Hence, $f_j(\bar{x}) \le 0,\,\forall j\in [m]$ follows from the continuity of $f_j$'s. For any $j\in [m]$, if $\bar{z}_j > 0$, then $z_j^k > \frac{\bar{z}_j}{2}$, as $k\in \cK$ is sufficiently large. It follows from $\max\big(-\frac{z_j^k}{\beta}, f_j(x^{k+1})\big) \to 0$ that $f_j(x^{k+1})\to 0$ as $\cK\ni k\to \infty$. Hence, $f_j(\bar{x})=0$. Therefore, $(\bar{x},\bar{z})$ satisfies \eqref{eq:kkt-cp}.

Thirdly, from the optimality of $x^{k+1}$, it holds that
$$\big\langle \nabla_x F(w^k), x^{k+1}\big\rangle + h(x^{k+1}) + \frac{\eta^k}{2}\|x^{k+1}-x^k\|^2 \le \big\langle \nabla_x F(w^k), x\big\rangle + h(x) + \frac{\eta^k}{2}\|x-x^k\|^2,\,\forall x.$$ 
Taking limit infimum over $k\in \cK$ on both sides of the above equation, we have from the lower semicontinuity of $h$ and continuity of $g$ and $f_j$'s that
$$\big\langle \nabla_x F(\bar{w}), \bar{x}\big\rangle + h(\bar{x})  \le \big\langle \nabla_x F(\bar{w}), x\big\rangle + h(x) + \frac{\eta^\infty}{2}\|x-\bar{x}\|^2,\,\forall x,$$
namely,
$$\bar{x}=\argmin_x \big\langle \nabla_x F(\bar{w}), x\big\rangle + h(x) + \frac{\eta^\infty}{2}\|x-\bar{x}\|^2.$$
Therefore, $\bar{w}$ satisfies \eqref{eq:kkt-1st} from the optimality condition of the above minimization problem, and thus $\bar{w}$ is a KKT point of \eqref{eq:ccp}.

Hence, \eqref{eq:tele-ineq2} holds with $w^*$ replaced by $\bar{w}$, and thus $w^k$ converges to $\bar{w}$ from Lemma \ref{lem:fejer}.
\end{proof}

\begin{theorem}[Sublinear convergence rate of LALM]\label{thm:rate-lalm}
Under Assumptions \ref{assump-kkt} and \ref{assump-lip}, let $\{w^k\}$ be the sequence generated from Algorithm \ref{alg:lalm} with any $x^0$, $y^0=0$ and $z^0=0$. If $\rho_y,\rho_z\in (0,\beta]$, then 
\begin{subequations}\label{eq:rate-lalm}
\begin{align}
|f_0(\bar{x}^{k+1})-f_0(x^*)| \le  \frac{\eta^\infty}{\eta^0(k+1)}\left(\frac{\eta^0}{2}\|x^0-x^*\|^2 + \frac{2\|y^*\|^2}{\rho_y} + \sum_{j=1}^m\frac{2(z_j^*)^2}{\rho_z}\right),\\
\|A\bar{x}^{k+1}-b\| + \sum_{j=1}^m [f_j(\bar{x}^{k+1})]_+ \le \frac{\eta^\infty}{\eta^0(k+1)}\left(\frac{\eta^0}{2}\|x^0-x^*\|^2 + \frac{(1+\|y^*\|)^2}{2\rho_y} + \sum_{j=1}^m\frac{(1+z_j^*)^2}{2\rho_z}\right),
\end{align}
\end{subequations}
where $\eta^\infty\le \bar{\eta}$ is the limit of $\{\eta^k\}$ and $\bar{x}^{k+1}=\sum_{t=0}^k \frac{x^{t+1}}{\sum_{t=0}^k\frac{1}{\eta^t}}$.
\end{theorem}

\begin{proof}
Letting $x=x^*$ in \eqref{eq:lalm-ineq2}, dividing by $\eta^k$, and summing it up give
\begin{align*}
\left(\sum_{t=0}^k\frac{1}{\eta^t}\right)\Phi(\bar{x}^{k+1};x^*,y,z)
\le \sum_{t=0}^k \frac{1}{\eta^t}\Phi(x^{t+1};x^*,y,z)
\le \frac{1}{\eta^0} \left[\frac{\eta^0}{2}\|x^0-x^*\|^2 + \frac{1}{2\rho_y}\|y\|^2+ \frac{1}{2\rho_z}\|z\|^2\right].
\end{align*}
Since $\sum_{t=0}^k\frac{1}{\eta^t}\ge \frac{k+1}{\eta^\infty}$, the desired results are obtained from Lemma \ref{lem:pre-rate} with $\alpha= \frac{\eta^\infty}{2(k+1)}\|x^0-x^*\|^2$, $c_1=\frac{\eta^\infty}{2\rho_y\eta^0(k+1)}$, and $c_2=\frac{\eta^\infty}{2\rho_z\eta^0(k+1)}$.
\end{proof}

Theorem \ref{thm:rate-lalm} implies that to reach an $\vareps$-optimal solution of \eqref{eq:ccp}, it is sufficient to evaluate the gradients of $g$ and $f_j, \, j\in[m]$ and proximal mapping of $h$ for $K$ times, where
$$K=\left\lfloor\frac{\eta^\infty}{\vareps\eta^0}\left(\frac{\eta^0}{2}\|x^0-x^*\|^2+\frac{\big[\max(1+\|y^*\|, 2\|y^*\|)\big]^2}{2\rho_y}+\sum_{j=1}^m\frac{\big[\max(1+z_j^*, 2z_j^*)\big]^2}{2\rho_z}\right)\right\rfloor.$$ 

\subsection{Local linear convergence of LALM for constrained smooth problems}\label{sec:local}
In this subsection, we assume that $h(x) = \iota_\cX(x)$ for a closed convex set $\cX$ and $g, f_1,\ldots,f_m$ are twice continuously differentiable. We show local linear convergence of Algorithm \ref{alg:lalm} under the following assumption.
\begin{assumption}\label{assump-scp}
There is a KKT point $w^*$ and a subset $J\subset [m]$ such that $x^*\in \mathrm{int}(\cX)$, and
\begin{enumerate}
\item $f_j(x^*) = 0, z_j^* >0, \,\forall j\in J$ and $f_j(x^*) <0, z_j^* =0,\,\forall j\not\in J$;
\item $x^\top \left(\nabla^2 g(x^*) + \sum_{j\in J} z_j^* \nabla^2 f_j(x^*)\right)x >0$ for any nonzero vector $x\in \Null(D^\top)$,
where $$D=\big[A^\top, \nabla f_1(x^*),\ldots, \nabla f_m(x^*)\big]$$
is column full-rank.
\end{enumerate}
\end{assumption}
When item 1 holds in the above assumption, we have
$$\nabla_x^2 \cL_\beta(w^*)=\nabla^2 g(x^*)+\beta A^\top A+\sum_{j\in J}\left(z_j^* \nabla^2 f_j(x^*)+\beta \nabla f_j(x^*)[\nabla f_j(x^*)]^\top\right),$$
and thus if in addition item 2 holds, then
$\nabla_x^2 \cL_\beta(w^*)$ is positive definite. We denote $\mu>0$ as its smallest eigenvalue. 

From the continuity of $\nabla _x^2 \cL_\beta$, we have the following result.
\begin{proposition}\label{prop:closeness}
There is $\gamma>0$ such that if $\max(\|x-x^*\|, \|y-y^*\|, \|z-z^*\|)\le \gamma$, then 
\begin{equation}\label{eq:close-result}
x\in \mathrm{int}(\cX);\, \nabla_x^2 \cL_\beta(w) \succeq \frac{\mu}{2}I; \, \beta f_j(x)+z_j >0, \forall j\in J;\, \beta f_j(x)+z_j < 0, \forall j\not\in J. 
\end{equation}
Hence, for any $x\in\cB_\gamma(x^*)$, 
\begin{equation}\label{eq:scv-ineq}
f_0(x)-f_0(x^*)+\langle y^*, Ax-b\rangle + \frac{\beta}{2}\|Ax-b\|^2+\Psi(x,z^*)-\Psi(x^*,z^*) \ge \frac{\mu}{4}\|x-x^*\|^2.
\end{equation}
\end{proposition}

The next lemma can be easily verified from the definition of $\Psi$. We omit its proof.
\begin{lemma}\label{lem:psi-sum}
If $x^{k+1}\in\cB_\gamma(x^*)$ and $z^k\in\cB_\gamma(z^*)$, then
\begin{equation}\label{eq:psi-sum}
\sum_{j\in J}(z_j^k-z_j^*) f_j(x^{k+1})=\Psi(x^{k+1},z^k)-\Psi(x^{k+1},z^*)-\Psi(x^*,z^k)+\Psi(x^*,z^*).
\end{equation}
\end{lemma}

From the update rule of $z$, we have following result.
\begin{lemma}
If $x^{k+1}\in\cB_\gamma(x^*)$ and $z^k\in\cB_\gamma(z^*)$, then
\begin{equation}\label{eq:z-zs}
\sum_{j=1}^m(z_j^{k+1}-z_j^*)\cdot\max\big(-\frac{z_j^k}{\beta}, f_j(x^{k+1})\big)\le \frac{1}{\rho_z}\|z^{k+1}-z^k\|^2+\sum_{j\in J}(z_j^k-z_j^*) f_j(x^{k+1}).
\end{equation}
\end{lemma}

\begin{proof}
When $x^{k+1}\in\cB_\gamma(x^*)$ and $z^k\in\cB_\gamma(z^*)$, it follows from \eqref{eq:close-result} that
$$\beta f_j(x^{k+1})+z_j^k >0, \forall j\in J;\, \beta f_j(x^{k+1})+z_j^k < 0, \forall j\not\in J. $$
Hence,
\begin{align*}
&\sum_{j=1}^m(z_j^{k+1}-z_j^*)\cdot\max\big(-\frac{z_j^k}{\beta}, f_j(x^{k+1}\big)\\
=&\sum_{j\in J}(z_j^k+\rho_z f_j(x^{k+1})-z_j^*)f_j(x^{k+1})+\sum_{j\not\in J}\big((1-\frac{\rho_z}{\beta})z_j^k\big)\big(-\frac{z_j^k}{\beta}\big)\\
\le &\sum_{j\in J}(z_j^k+\rho_z f_j(x^{k+1})-z_j^*)f_j(x^{k+1})+\rho_z\sum_{j\not\in J}\big(\frac{z_j^k}{\beta}\big)^2\\
=&\frac{1}{\rho_z}\|z^{k+1}-z^k\|^2+\sum_{j\in J}(z_j^k-z_j^*) f_j(x^{k+1}),
\end{align*}
which completes the proof.
\end{proof}

Let 
$\underline{\eta}=L_g+\beta\|A\|^2+\sum_{j=1}^m\beta B_j^2.
$
Then we have the next theorem.
\begin{theorem}\label{thm:pre-lin}
Let $\{w^k\}$ be the sequence generated from Algorithm \ref{alg:lalm} with $w^0$ satisfying the following condition:
\begin{equation}\label{eq:init-cond}
\|x^0-x^*\|^2+\frac{1}{\rho_y \underline{\eta}}\|y^0-y^*\|^2+\frac{1}{\rho_z \underline{\eta}}\|z^0-z^*\|^2\le \gamma^2\cdot\min\big(1,\frac{1}{\rho_y\bar{\eta}},\frac{1}{\rho_z\bar{\eta}}\big),
\end{equation}
where $\gamma$ is given in Proposition \ref{prop:closeness}.
For any $\theta\in(0,1)$, if $0<\rho_y\le \beta$ and $0<\rho_z\le \beta(1-\theta)$, 
then for any $k$, it holds that
\begin{align}\label{eq:loc-ineq1}
&\frac{\theta\mu}{4}\|x^{k+1}-x^*\|^2+\frac{\eta^k}{2}\|x^{k+1}-x^*\|^2-\frac{\beta}{2}\|r^{k+1}\|^2+\frac{1}{2\rho_y}\|y^{k+1}-y^*\|^2+\frac{1}{2\rho_z}\|z^{k+1}-z^*\|^2\cr
&+ \frac{1}{2}\|x^{k+1}-x^k\|_{(\eta^k-L_g-L_\Psi^k)I-\beta A^\top A}^2+\left(\beta\big(1-\frac{\theta}{2}\big)-\frac{\rho_y}{2}\right)\|r^{k+1}\|^2\cr
\le &\frac{\eta^k}{2}\|x^k-x^*\|^2-\frac{\beta}{2}\|r^k\|^2+\frac{1}{2\rho_y}\|y^k-y^*\|^2+\frac{1}{2\rho_z}\|z^k-z^*\|^2.
\end{align}
\end{theorem}

\begin{proof}
We first note $\max(\|x^k-x^*\|, \|y^k-y^*\|,\|z^k-z^*\|)\le \gamma,\,\forall k\ge0$ from \eqref{eq:tele-ineq2}, \eqref{eq:init-cond}, and the following inequality
\begin{align*}
&~\min\big(1,\frac{1}{\rho_y\bar{\eta}},\frac{1}{\rho_z\bar{\eta}}\big)(\|x^k-x^*\|^2+ \|y^k-y^*\|^2+ \|z^k-z^*\|^2)\\
\le &~\|x^k-x^*\|^2+ \frac{1}{\rho_y\eta^k}\|y^k-y^*\|^2+ \frac{1}{\rho_z\eta^k}\|z^k-z^*\|^2\\
\le &~ \|x^0-x^*\|^2+\frac{1}{\rho_y \underline{\eta}}\|y^0-y^*\|^2+\frac{1}{\rho_z \underline{\eta}}\|z^0-z^*\|^2.
\end{align*}

Adding \eqref{eq:lalm-yterm} with $y=y^*$, \eqref{eq:lalm-zterm} with $z=z^*$, $\theta$ times of \eqref{eq:psi-sum} and \eqref{eq:z-zs}, and $1-\theta$ times of \eqref{eq:lalm-ineq-z} to \eqref{eq:lalm-ineq-a} with $(x,y)=(x^*,y^*)$, we have by rearranging terms that
\begin{align*}
&\theta\left[f_0(x^{k+1})-f_0(x^*)+\langle y^*, r^{k+1}\rangle + \frac{\beta}{2}\|r^{k+1}\|^2+\Psi(x^{k+1},z^*)-\Psi(x^*,z^*)\right]\cr
&+(1-\theta)\left[f_0(x^{k+1})-f_0(x^*)+\langle y^*, r^{k+1}\rangle+\sum_{j=1}^m z_j^* f_j(x^{k+1})-\Psi(x^*,z^k)\right]\cr
&+\frac{\eta^k}{2}\|x^{k+1}-x^*\|^2-\frac{\beta}{2}\|r^{k+1}\|^2+\frac{1}{2\rho_y}\|y^{k+1}-y^*\|^2+\frac{1}{2\rho_z}\|z^{k+1}-z^*\|^2\cr
& + \frac{1}{2}\|x^{k+1}-x^k\|_{(\eta^k-L_g-L_\Psi^k)I-\beta A^\top A}^2+\left(\beta\big(1-\frac{\theta}{2}\big)-\frac{\rho_y}{2}\right)\|r^{k+1}\|^2+\frac{1}{2\rho_z}\left(\frac{\beta(1-\theta)}{\rho_z}-1\right)\|z^{k+1}-z^k\|^2\cr
\le &\frac{\eta^k}{2}\|x^k-x\|^2-\frac{\beta}{2}\|r^k\|^2+\frac{1}{2\rho_y}\|y^k-y^*\|^2+\frac{1}{2\rho_z}\|z^k-z^*\|^2.
\end{align*}
From \eqref{eq:opt}, \eqref{eq:scv-ineq}, the above inequality, and $\Psi(x^*, z^k)\le 0$, the desired result follows.
\end{proof}

In addition, we can bound $\|y^k-y^*\|^2$ and $\|z^k-z^*\|^2$ by $x$-terms. 
\begin{lemma}\label{lem:bd-yz-x}
Let $\nu>0$ be the smallest eigenvalue of $D^\top D$. Under the assumption of Theorem \ref{thm:pre-lin}, we have
\begin{equation}\label{eq:bd-yz-x}
\begin{aligned}
&\nu \big(\|y^k-y^*\|^2+\|z^k-z^*\|^2\big)\\
\le & \left(4 L_g^2+8|J|\sum_{j\in J}(\beta^2 B_j^4+ |z_j^k|^2 L_j^2)\right)\|x^k-x^*\|^2 + 4\beta^2\|A\|^2\|r^k\|^2+4\bar{\eta}^2\|x^{k+1}-x^k\|^2.
\end{aligned}
\end{equation}
\end{lemma}

\begin{proof}
Note that $$\nabla_x \Psi(x^k,z^k)=\sum_{j=1}^m[\beta f_j(x^k)+z_j^k]_+\nabla f_j(x^k)=\sum_{j\in J}(\beta f_j(x^k)+z_j^k)\nabla f_j(x^k).$$
Hence, from the update of $x$ and the fact $x^{k+1}\in\mathrm{int}(\cX)$, it follows
$$\nabla g(x^k)+A^\top y^k+\beta A^\top r^k + \sum_{j\in J}(\beta f_j(x^k)+z_j^k)\nabla f_j(x^k)+\eta^k (x^{k+1}-x^k)=0.$$
In addition, since $x^*\in \mathrm{int}(\cX)$, it holds that
$$\nabla g(x^*)+A^\top y^* + \sum_{j\in J}z_j^* \nabla f_j(x^*)=0.$$
From the above two equations, it follows that
{\small
\begin{align}\label{eq:bd-yz-x-1}
&\big\|D[y^k; z^k]-D[y^*; z^*]\big\|^2\cr
=&\big\|\nabla g(x^k)-\nabla g(x^*)+\beta A^\top r^k+\sum_{j\in J}(\beta f_j(x^k)+z_j^k)\nabla f_j(x^k)-\sum_{j\in J}z_j^k \nabla f_j(x^*)+\eta^k (x^{k+1}-x^k)\big\|^2\cr
\le & 4\left(\|\nabla g(x^k)-\nabla g(x^*)\|^2+\|\beta A^\top r^k\|^2+\big\|\sum_{j\in J}\big[(\beta f_j(x^k)+z_j^k)\nabla f_j(x^k)-z_j^k \nabla f_j(x^*)\big]\big\|^2+\|\eta^k (x^{k+1}-x^k)\|^2\right)\cr
\le &4 L_g^2\|x^k-x^*\|^2 + 4\beta^2\|A\|^2\|r^k\|^2+4\bar{\eta}^2\|x^{k+1}-x^k\|^2+4\big\|\sum_{j\in J}\big[(\beta f_j(x^k)+z_j^k)\nabla f_j(x^k)-z_j^k \nabla f_j(x^*)\big]\big\|^2.
\end{align}
}
Note $f_j(x^*)=0,\,\forall j\in J$. Hence,
\begin{align*}
&\big\|\sum_{j\in J}\big[(\beta f_j(x^k)+z_j^k)\nabla f_j(x^k)-z_j^k \nabla f_j(x^*)\big]\big\|^2\cr
\le& |J|\sum_{j\in J}\big\|(\beta f_j(x^k)+z_j^k)\nabla f_j(x^k)-z_j^k \nabla f_j(x^*)\big\|^2\cr
=&|J|\sum_{j\in J}\big\|\beta \big(f_j(x^k)-f_j(x^*)\big)\nabla f_j(x^k)+z_j^k\nabla f_j(x^k)-z_j^k \nabla f_j(x^*)\big\|^2\cr
\le & 2|J|\sum_{j\in J}(\beta^2 B_j^4+ |z_j^k|^2 L_j^2) \|x^k-x^*\|^2.
\end{align*}
Plugging in the above inequality into \eqref{eq:bd-yz-x-1} and noting $\nu\big\|[y^k; z^k]-[y^*; z^*]\big\|^2\le \big\|D[y^k; z^k]-D[y^*; z^*]\big\|^2$, we obtain the desired result.
\end{proof}

If necessary, taking a smaller $\gamma$, we can assume
\begin{equation}\label{eq:gap-eta}
\left|\sum_{j\in J}L_j(\beta f_j(\hat{x})+\hat{z}_j)-\sum_{j\in J}L_j(\beta f_j(\tilde{x})+\tilde{z}_j)\right| \le \frac{\mu}{8},\,\forall \hat{x},\tilde{x}\in \cB_\gamma(x^*), \forall \hat{z},\tilde{z}\in\cB_\gamma(z^*).
\end{equation}
Then we have the local linear convergence of Algorithm \ref{alg:lalm} as follows.
\begin{theorem}[Local linear convergence]\label{thm:loc-lin}
Under Assumptions \ref{assump-lip} and \ref{assump-scp}, let $\{w^k\}$ be the sequence generated from Algorithm \ref{alg:lalm} with $w^0$ satisfying \eqref{eq:init-cond}, $\rho_y=\rho_z=\frac{\beta}{2}$, and $\delta>0$. Let $$C=L_g^2+2|J|\left(\sum_{j\in J}\beta^2 B_j^4+ 2L_{\max}^2(|z^*|^2+\gamma^2) \right),$$
where $L_{\max}=\max_j L_j$. For any $\alpha>0$ such that
\begin{align}
\alpha <\min\left(\frac{\mu}{8C},\, \frac{\delta}{\bar{\eta}^2},\, \frac{1}{\beta\|A\|^2}\right),
\end{align} 
it holds $\phi(x^{k+1},y^{k+1},z^{k+1}) \le \sigma\cdot \phi(x^k,y^k,z^k)$, where
$$\phi(x^k,y^k,z^k)=\big(\frac{\mu}{16}+\frac{\eta^k}{2}\big)\|x^{k}-x^*\|^2+\frac{1}{\beta}\big(\|y^k-y^*\|^2+\|z^k-z^*\|^2\big)$$
and
$$\sigma=\max\left(\frac{\alpha C+\bar{\eta}}{\frac{\mu}{8}+\bar{\eta}}, 1-\frac{\alpha\beta\nu}{8}\right)<1.$$
\end{theorem}

\begin{proof}
Adding $\frac{\alpha}{8}$ of \eqref{eq:bd-yz-x} to \eqref{eq:loc-ineq1} with $\theta=\frac{1}{2}$, and noting $\alpha \bar{\eta}^2 I \preceq (\eta^k-L_g-L_\Psi^k)I -\beta A^\top A,\,\forall k$ and $\alpha\beta^2\|A\|^2\le \beta$ gives
\begin{align}\label{eq:loc-lin-ineq1}
&\big(\frac{\mu}{8}+\frac{\eta^k}{2}\big)\|x^{k+1}-x^*\|^2+\frac{1}{\beta}\big(\|y^{k+1}-y^*\|^2+\|z^{k+1}-z^*\|^2\big)\nonumber\\
\le &\big(\frac{\alpha C}{2}+\frac{\eta^k}{2}\big)\|x^k-x^*\|^2+\big(\frac{1}{\beta}-\frac{\alpha\nu}{8}\big)\big(\|y^k-y^*\|^2+\|z^k-z^*\|^2\big).
\end{align}
Let 
$$\eta_{\max}=\delta+L_g+\beta\|A\|^2+\sum_{j=1}^m\beta B_j^2+\max_{\substack{x\in\cB_\gamma(x^*)\\ z\in\cB_\gamma(z^*)}}\sum_{j\in J}L_j(\beta f_j(x)+z_j).$$
From the setting of $\eta^k$, we have that
$$\eta^k=\max(L_F^0,L_F^1, \ldots, L_F^k) + \delta \le \eta_{\max},\,\forall k.$$
In addition, \eqref{eq:gap-eta} indicates $\eta^{k+1}-\eta^k \le \eta_{\max} - (L_F^k+\delta) \le \frac{\mu}{8}$. Hence, 
\begin{equation}\label{eq:eta-k1}
\frac{\mu}{16}+\frac{\eta^{k+1}}{2}\le \frac{\mu}{8}+\frac{\eta^k}{2}.
\end{equation} 
Since $\frac{\alpha C}{2}+\frac{\eta^k}{2}\le \sigma \big(\frac{\mu}{16}+\frac{\eta^k}{2}\big)$ and $\frac{1}{\beta}-\frac{\alpha\nu}{8}\le  \frac{\sigma}{\beta}$, we have the desired result from \eqref{eq:loc-lin-ineq1}, \eqref{eq:eta-k1}, and the definition of $\phi$.
\end{proof}

\begin{remark}
In Theorem \ref{thm:loc-lin}, the setting of $\rho_y=\rho_z=\frac{\beta}{2}$ is for simplicity of the analysis. The local linear convergence can be obtained for any $\rho_y,\rho_z\in (0,\beta)$. Therefore, from Theorems \ref{thm:convg-lalm} and \ref{thm:loc-lin}, the algorithm may eventually converge linearly. This phenomenon is observed from our numerical experiments; see Figures \ref{fig:bpdn} and \ref{fig:qcqp}.
\end{remark}

\section{Block linearized augmented Lagrangian method}\label{sec:blalm}
In this section, we assume that in \eqref{eq:ccp}, $x$ can be partitioned into $n$ disjoint blocks and the non-differentiable part $h(x)$ is separable, i.e., 
 $$x=(x_1,x_2,\ldots,x_n),\quad h(x)=\sum_{i=1}^n h_i(x_i).$$ 
 Correspondingly, $A$ can be written as the block matrix format $[A_1,\ldots,A_n]$. 
 
 \subsection{Algorithm}
 Towards a solution of the block structured problem, we propose a block linearized augmented Lagrangian method (BLALM). At each iteration, it randomly picks one block primal variable to update and then immediately renews the multipliers. The method is summarized in Algorithm \ref{alg:blalm}.

\begin{algorithm}
\caption{Block linearized augmented Lagrangian method for \eqref{eq:ccp}}\label{alg:blalm}
\DontPrintSemicolon
\textbf{Initialization:} choose $x^0, y^0, z^0$ and $\beta,\rho_y,\rho_z,\veta=[\eta_1,\cdots,\eta_n]$; let $r^0=Ax^0-b$\;
\For{$k=0,1,\ldots$}{
Pick $i_k\in [n]$ uniformly at random and perform the updates
\begin{subequations}
\begin{align}
x^{k+1}_i=&~\left\{\begin{array}{ll}\underset{x_i}\argmin ~h_i(x_i) + \langle \nabla_{x_i} F_\beta(w^k), x_i\rangle + \frac{\eta_i}{2}\|x_i-x^k_i\|^2,& \text{ if }i=i_k\\[0.1cm]
x_i^k, & \text{ if }i\neq i_k\end{array}\right.\label{eq:blalm-x}\\[0.1cm]
r^{k+1} = &~ r^k + A_{i_k}(x_{i_k}^{k+1}-x_{i_k}^k),\nonumber\\[0.1cm]
y^{k+1}=&~y^k + \rho_y r^{k+1}\label{eq:blalm-y},\\
z^{k+1}_j=&~z_j^k+\rho_z \cdot\max\left(-\frac{z_j^k}{\beta}, f_j(x^{k+1})\right), j=1,\ldots, m.\label{eq:blalm-z}
\end{align}
\end{subequations}
}
\end{algorithm}

To make Algorithm \ref{alg:blalm} efficient, we require \eqref{eq:ccp} to have the so-called coordinate friendly structure \cite{peng2016cf}. Roughly speaking, computing all $n$ block partial gradients $\nabla_{x_i} F_\beta$ has nearly the same complexity as a full gradient evaluation. In addition, $f(x^{k+1})$ can be easily calculated from $x^k$, $f(x^k)$ and the change of $x_{i_k}$.

We let $\ell_i^k$ be the Lipschitz constant of $\nabla_{x_i} g(x)+\nabla_{x_i}\Psi(x,z^k)$ with respect to $x_i$ for every $i=1,\ldots,n$ and $\vell^k=[\ell_1^k,\cdots,\ell_n^k]$. In general, $\ell_i^k$ can be significantly smaller than the Lipschitz constant of $\nabla g(x)+\nabla_x \Psi(x,z)$, and thus a larger stepsize can be made if a single block is updated instead of all blocks.


\subsection{Convergence analysis}

To show the convergence results of Algorithm \ref{alg:blalm}, we first establish a fundamental result that is similar to Theorem \ref{thm:1iter}.

\begin{theorem}[One-iteration result of BLALM]
Let $\{w^k\}$ be the sequence generated from Algorithm \ref{alg:blalm}. Then for any $x$ such that $Ax=b$ and $f_j(x) \le 0,\forall j\in [m]$, it holds
\begin{align}\label{eq:b-1iter}
&~\EE_{i_k}\left[f_0(x^{k+1})-f_0(x)+\langle y^{k+1}, r^{k+1}\rangle + (\beta-\rho_y)\|r^{k+1}\|^2+\Psi_\beta(x^{k+1},z^k)\right]\cr
&~+\frac{1}{2}\EE_{i_k}\left[\|x^{k+1}-x\|_{\veta}^2-\|x^k-x\|_{\veta}^2+\|x^{k+1}-x^k\|_{\veta-\vell^k}^2\right]-\frac{\beta}{2}\EE_{i_k}\big[\|r^{k+1}\|^2-\|r^k\|^2+\|x^{k+1}-x^k\|_{A^\top A}^2\big]\cr
\le &~ \big(1-\frac{1}{n}\big)\big[f_0(x^k)-f_0(x)+\langle y^k, r^k\rangle + \beta\|r^k\|^2+\Psi_\beta(x^k,z^k)\big].
\end{align}
\end{theorem}

\begin{proof}
From the update of $x_{i_k}$, we have
\begin{equation}\label{eq:opt-x-ik}
0\in\partial h_{i_k}(x_{i_k}^{k+1})+\nabla_{x_{i_k}}g(x^k)+A_{i_k}^\top(y^k+\beta r^k)+\nabla_{x_{i_k}} \Psi_\beta(x^k, z^k)+\eta_{i_k}(x_{i_k}^{k+1}-x^k_{i_k}).
\end{equation}
Note that for any $x$,
\begin{align}\label{eq:b-h-term}
\EE_{i_k}\big\langle x_{i_k}^{k+1}-x_{i_k}, \tilde{\nabla} h_{i_k}(x_{i_k}^{k+1})\big\rangle \ge & ~\EE _{i_k} [h_{i_k}(x_{i_k}^{k+1})-h_{i_k}(x_{i_k})]\cr
=&~\EE _{i_k} [h_{i_k}(x_{i_k}^{k+1})-h_{i_k}(x_{i_k}^k)+h_{i_k}(x_{i_k}^k)-h_{i_k}(x_{i_k})]\cr
=&~\EE_{i_k}[h(x^{k+1}-h(x^k)] + \frac{1}{n} [h(x^k)-h(x)],
\end{align}
and
\begin{align}\label{eq:b-psi-term}
&~\EE_{i_k}\big\langle x_{i_k}^{k+1}-x_{i_k},\nabla_{x_{i_k}}g(x^k)+\nabla_{x_{i_k}} \Psi_\beta(x^k, z^k)\big\rangle\cr
=&~\EE_{i_k}\big\langle x_{i_k}^{k+1}-x_{i_k}^k+x_{i_k}^k-x_{i_k},\nabla_{x_{i_k}}g(x^k)+\nabla_{x_{i_k}} \Psi_\beta(x^k, z^k)\big\rangle\cr
\ge&~ \EE_{i_k}\big[g(x^{k+1})+\Psi_\beta(x^{k+1}, z^k)-g(x)-\Psi_\beta(x^k, z^k)-\frac{1}{2}\|x^{k+1}-x^k\|_{\vell^k}^2\big]\\
&~ + \frac{1}{n} \big[g(x^k)+\Psi_\beta(x^k, z^k)-g(x)-\Psi_\beta(x, z^k)\big]\nonumber.
\end{align}
In addition, for any $x$ such that $Ax=b$, we have from \cite[Lemma 3.2]{xu2017async-pd} that
\begin{align}\label{eq:b-y-term}
\EE_{i_k}\big\langle x_{i_k}^{k+1}-x_{i_k}, A_{i_k}^\top(y^k+\beta r^k)\big\rangle=&-\big(1-\frac{1}{n}\big)(\langle y^k, r^k\rangle+\beta\|r^k\|^2)+\EE_{i_k}\big[\langle y^k, r^{k+1}\rangle + \beta \|r^{k+1}\|^2\big]\nonumber\\
&~ -\frac{\beta}{2}\EE_{i_k}\big[\|r^{k+1}\|^2-\|r^k\|^2+\|x^{k+1}-x^k\|_{A^\top A}^2\big].
\end{align}
Furthermore,
\begin{equation}\label{eq:b-eta-term}
\EE_{i_k}\langle x_{i_k}^{k+1}-x_{i_k}, \eta_{i_k}(x_{i_k}^{k+1}-x^k_{i_k})\rangle = \frac{1}{2}\EE_{i_k}\big[\|x^{k+1}-x\|_{\veta}^2-\|x^k-x\|_{\veta}^2+\|x^{k+1}-x^k\|_{\veta}^2\big].
\end{equation}
Adding \eqref{eq:b-h-term} through \eqref{eq:b-eta-term}, we have from \eqref{eq:opt-x-ik} that
\begin{align*}
&~\EE_{i_k}\left[f_0(x^{k+1})-f_0(x)+\langle y^k, r^{k+1}\rangle + \beta\|r^{k+1}\|^2+\Psi_\beta(x^{k+1},z^k)\right]\cr
&~+\frac{1}{2}\EE_{i_k}\left[\|x^{k+1}-x\|_{\veta}^2-\|x^k-x\|_{\veta}^2+\|x^{k+1}-x^k\|_{\veta-\vell^k}^2\right]-\frac{\beta}{2}\EE_{i_k}\big[\|r^{k+1}\|^2-\|r^k\|^2+\|x^{k+1}-x^k\|_{A^\top A}^2\big]\cr
\le &~ \big(1-\frac{1}{n}\big)\big[f_0(x^k)-f_0(x)+\langle y^k, r^k\rangle + \beta\|r^k\|^2+\Psi_\beta(x^k,z^k)\big]+\frac{1}{n}\Psi_\beta(x,z^k).
\end{align*}
Since $y^{k+1}=y^k+\rho_y r^{k+1}$ and $\Psi_\beta(x,z^k)\le 0$, \eqref{eq:b-1iter} is obtained from the above inequality.
\end{proof}

We also need the next lemma.
\begin{lemma}\label{lem:psi-k}
For any $\rho_z\le\beta$,
\begin{align}\label{eq:psi-k}
-\frac{1}{\rho_z}\|z^{k+1}-z^k\|^2\le \Psi_\beta(x^{k+1},z^k)-\Psi_\beta(x^{k+1},z^{k+1}) 
\end{align}
\end{lemma}
\begin{proof}
Since $\rho_z\le\beta$, we have $z_j^k\ge 0,\forall k$.
Let 
\begin{subequations}\label{eq:def-Jk}
\begin{align}
&J_1^k = \{j\in [m]: \beta f_j(x^{k+1})+z_j^k \ge 0,\, \beta f_j(x^{k+1})+z_j^{k+1} \ge 0\},\\
&J_2^k = \{j\in [m]: \beta f_j(x^{k+1})+z_j^k \ge 0,\, \beta f_j(x^{k+1})+z_j^{k+1} < 0\},\\
&J_3^k = \{j\in [m]: \beta f_j(x^{k+1})+z_j^k < 0\}.
\end{align}
\end{subequations}
For any $j\in J_1^k \cup J_2^k$, $z_j^{k+1}=z_j^k+\rho_z f_j(x^{k+1})$, and for any $j\in J_3^k$, $z_j^{k+1}=\big(1-\frac{\rho_z}{\beta}\big)z_j^k$ and $\beta f_j(x^{k+1})+z_j^{k+1} < 0.$ Hence,
\begin{align}\label{eq:diff-Psik}
&\Psi_\beta(x^{k+1},z^k)-\Psi_\beta(x^{k+1},z^{k+1})\cr
=&\sum_{j\in J_1^k\cup J_2^k}\left(z_j^k f_j(x^{k+1})+\frac{\beta}{2}[f_j(x^{k+1})]^2\right)-\sum_{j\in J_3^k}\frac{(z_j^k)^2}{2\beta}\cr
&-\sum_{j\in J_1^k}\left(z_j^{k+1} f_j(x^{k+1})+\frac{\beta}{2}[f_j(x^{k+1})]^2\right)+\sum_{j\in J_2^k\cup J_3^k}\frac{(z_j^{k+1})^2}{2\beta}\cr
=&-\sum_{j\in J_1^k}\rho_z[f_j(x^{k+1})]^2 +\sum_{j\in J_2^k}\left(z_j^k f_j(x^{k+1})+\frac{\beta}{2}[f_j(x^{k+1})]^2+\frac{(z_j^{k+1})^2}{2\beta}\right) \\
&-\sum_{j\in J_3^k}\frac{(z_j^k)^2-(z_j^{k+1})^2}{2\beta}\nonumber
\end{align}

For $j\in J_2^k$, we have
\begin{align}\label{eq:ineqJ2}
&~z_j^k f_j(x^{k+1})+\frac{\beta}{2}[f_j(x^{k+1})]^2+\frac{(z_j^{k+1})^2}{2\beta}\cr
=&~z_j^k f_j(x^{k+1})+\frac{\beta}{2}[f_j(x^{k+1})]^2+\frac{(z_j^k+\rho_z f_j(x^{k+1}))^2}{2\beta}\cr
=&~\frac{(z_j^k)^2}{2\beta}+ \big(1+\frac{\rho_z}{\beta}\big)z_j^k f_j(x^{k+1})+\big(\frac{\beta}{2}+\frac{\rho_z^2}{2\beta}\big)[f_j(x^{k+1})]^2\cr
\ge&~-\rho_z[f_j(x^{k+1})]^2,
\end{align}
where the inequality follows from the Young's inequality. For $j\in J_3^k$, we have 
\begin{align}\label{eq:ineqJ3}
-\frac{(z_j^k)^2-(z_j^{k+1})^2}{2\beta}=-\left(1-\big(1-\frac{\rho_z}{\beta}\big)^2\right)\frac{(z_j^k)^2}{2\beta}\ge -\frac{\rho_z}{\beta^2}(z_j^k)^2.
\end{align}
Plugging \eqref{eq:ineqJ2} and \eqref{eq:ineqJ3} into \eqref{eq:diff-Psik} gives
\begin{equation*}
\Psi_\beta(x^{k+1},z^k)-\Psi_\beta(x^{k+1},z^{k+1}) \ge -\sum_{j\in J_1^k\cup J_2^k}\rho_z[f_j(x^{k+1})]^2-\sum_{j\in J_3^k}\frac{\rho_z}{\beta^2}(z_j^k)^2=-\frac{1}{\rho_z}\|z^{k+1}-z^k\|^2,
\end{equation*}
which completes the proof.
\end{proof}

The following results are easy to show from the Young's inequality and the update rule of $z$.
\begin{lemma}\label{lem:yz-k-bd}
For any $y$ and $z\ge0$, 
\begin{equation}\label{eq:y-k-bd}
\langle y, r^{k+1}\rangle\le \langle y^k, r^{k+1}\rangle+\frac{\beta}{2}\|r^{k+1}\|^2+\frac{1}{2\beta}\|y^k-y\|^2,
\end{equation}
and
\begin{equation}\label{eq:z-k-bd}
0\le \Psi_\beta(x^{k+1},z^k)-\sum_{j=1}^m z_j f_j(x^{k+1})+\frac{1}{2\beta}\|z^k-z\|^2
\end{equation}
\end{lemma}

\begin{proof}
The inequality in \eqref{eq:y-k-bd} directly follows from the Young's inequality.

Let $J_+^k=J_1^k\cup J_2^k$ and $J_-^k=J_3^k$, where $J_1^k, J_2^k$ and $J_3^k$ are defined in \eqref{eq:def-Jk}. Then
\begin{align*}
&~\Psi_\beta(x^{k+1},z^k)-\sum_{j=1}^m z_j f_j(x^{k+1})+\frac{1}{2\beta}\|z^k-z\|^2\cr
=&~\sum_{j\in J_+^k}\left[(z_j^k-z_j)f_j(x^{k+1})+\frac{\beta}{2}[f_j(x^{k+1})]^2+\frac{1}{2\beta}(z_j^k-z_j)^2\right]+\sum_{j\in J_-^k}\left[-\frac{(z_j^k)^2}{2\beta}-z_jf_j(x^{k+1})+\frac{1}{2\beta}(z_j^k-z_j)^2\right]\cr
\ge &~\sum_{j\in J_-^k}\left[-\frac{(z_j^k)^2}{2\beta}-z_jf_j(x^{k+1})+\frac{1}{2\beta}(z_j^k-z_j)^2\right]\ge \sum_{j\in J_-^k} \frac{1}{2\beta}(z_j)^2,
\end{align*}
where the first inequality follows from the Young's inequality, and the second one holds because $f_j(x^{k+1})\le -\frac{z_j^k}{\beta},\forall j\in J_-^k$ and $z_j\ge 0,\forall j$. This completes the proof.
\end{proof}

Using the previous establish results, we are now able to show the convergence rate of Algorithm \ref{alg:blalm}.

\begin{theorem}[Sublinear convergence of BLALM]\label{thm:b-rate}
Under Assumptions \ref{assump-kkt} and \ref{assump-lip}, let $\{w^k\}$ be the sequence from Algorithm \ref{alg:blalm} with $y^0=0$ and $z^0=0$. Assume $\ell_i^k$ is upper bounded by $\bar{\ell}_i$ for any $i\in[n]$ and any $k$. If $\rho_y\in (0,\frac{\beta}{n}]$, $\rho_z\in (0,\frac{\beta}{2n}]$, and $\eta_i\ge \bar{\ell}_i+\beta\|A_i\|^2,\,\forall i\in [n]$, then
\begin{subequations}\label{eq:rate-blalm} 
\begin{align}
\big|\EE[f_0(\bar{x}^{k+1})-f_0(x^*)]\big| \le \frac{1}{1+\frac{k}{n}}\left(C_{x^0}+\frac{2\|y^*\|^2}{n\rho_y}+\frac{2\|z^*\|^2}{n\rho_z}\right),\label{eq:rate-blalm-obj}\\
\EE\left[\|A\bar{x}^{k+1}-b\| + \sum_{j=1}^m [f_j(\bar{x}^{k+1})]_+\right] \le \frac{1}{1+\frac{k}{n}}\left(C_{x^0} + \frac{(1+\|y^*\|)^2}{2n\rho_y} + \frac{1}{2n\rho_z}\sum_{j=1}^m\big(1+z_j^*\big)^2\right), \label{eq:rate-blalm-res}
\end{align}
\end{subequations}
where $\bar{x}^{k+1}=\frac{1}{1+\frac{k}{n}}\sum_{t=0}^k x^{t+1}$, and
$$C_{x^0}=\big(1-\frac{1}{n}\big)\left[f_0(x^0)-f_0(x^*) + \frac{\beta}{2}\left(\|r^0\|^2+\sum_{j=1}^m [f_j(x^0)]_+^2\right)\right]+\frac{1}{2}\|x^0-x^*\|_{\veta}^2.$$
\end{theorem}

\begin{proof}
Since $\eta_i\ge \bar{\ell}_i+\beta\|A_i\|^2,\,\forall i\in [n]$ and $x^{k+1}_i=x^k_i,\,\forall i\neq i_k$, it holds 
$$\|x^{k+1}-x^k\|^2_{\veta-\vell^k}\ge \beta\|x^{k+1}-x^k\|^2_{A^\top A}.$$ Hence, taking expectation on both sides of \eqref{eq:b-1iter} with $x=x^*$ and summing it up give
\begin{align}\label{eq:b-1iter-sum}
&~\EE\left[f_0(x^{k+1})-f_0(x^*)+\langle y^k, r^{k+1}\rangle +\Psi_\beta(x^{k+1},z^k)\right]+\frac{\beta}{2}\EE\|r^{k+1}\|^2\cr
&~+\frac{1}{n}\sum_{t=0}^{k-1}\EE\left[f_0(x^{t+1})-f_0(x^*)+\langle y^{t+1}, r^{t+1}\rangle +\Psi_\beta(x^{t+1},z^t)\right]+ \big(\frac{\beta}{n}-\rho_y\big)\sum_{t=0}^{k-1}\EE\|r^{t+1}\|^2\cr
&~+\big(1-\frac{1}{n}\big)\sum_{t=0}^{k-1}\EE\big[\Psi_\beta(x^{t+1},z^t)-\Psi_\beta(x^{t+1},z^{t+1})\big]+\frac{1}{2}\EE\|x^{k+1}-x^*\|_{\veta}^2\cr
\le &~ \big(1-\frac{1}{n}\big)\left[f_0(x^0)-f_0(x^*)+\langle y^0, r^0\rangle + \frac{\beta}{2}\|r^0\|^2+\Psi_\beta(x^0,z^0)\right]+\frac{1}{2}\|x^0-x^*\|_{\veta}^2-\frac{\beta}{2}\|r^0\|^2,
\end{align}
where in the first line, we have used $y^{k+1}=y^k-\rho_y r^{k+1}$.
Summing \eqref{eq:lalm-yterm}, \eqref{eq:lalm-zterm}, and \eqref{eq:lalm-ineq-z} gives
\begin{align*}
&~\frac{1}{2\rho_y}\left[\|y^k-y\|^2-\|y^0-y\|^2+\sum_{t=0}^{k-1}\|y^{t+1}-y^t\|^2\right] - \sum_{t=0}^{k-1}\langle y^{t+1}-y, r^{t+1}\rangle\cr
&~+\frac{1}{2\rho_z}\left[\|z^k-z\|^2-\|z^0-z\|^2+\sum_{t=0}^{k-1}\|z^{t+1}-z^t\|^2\right]+\frac{\beta-2\rho_z}{2\rho_z^2}\sum_{t=0}^{k-1}\|z^{t+1}-z^t\|^2\cr
\le &~\sum_{t=0}^{k-1}\left[\Psi_\beta(x^{t+1},z^t)-\sum_{j=1}^m z_j f_j(x^{t+1})\right].
\end{align*}
Since $y^0=0$ and $z^0=0$, adding $\frac{1}{n}$ of the above inequality to \eqref{eq:b-1iter-sum}, using Lemma \ref{lem:psi-k}, and noting $\frac{\beta}{n}\ge \rho_y$, $\frac{\beta-\rho_z}{2n\rho_z^2}\ge\frac{1-\frac{1}{n}}{\rho_z}$ from the choice of $\rho_y,\rho_z$, we have
\begin{align}\label{eq:b-1iter-yz-sum}
&~\EE\left[f_0(x^{k+1})-f_0(x^*)+\langle y^k, r^{k+1}\rangle +\Psi_\beta(x^{k+1},z^k)\right]+\frac{\beta}{2}\EE\|r^{k+1}\|^2+\frac{1}{2n\rho_y}\EE\|y^k-y\|^2\cr
&~+\frac{1}{2}\EE\|x^{k+1}-x^*\|_{\veta}^2+\frac{1}{2n\rho_z}\EE\|z^k-z\|^2+\frac{1}{n}\sum_{t=0}^{k-1}\EE [\Phi(x^{t+1}; x^*, y,z)]\cr
\le &~ \big(1-\frac{1}{n}\big)\left[f_0(x^0)-f_0(x^*) + \beta\|r^0\|^2+\Psi_\beta(x^0,0)\right]+\frac{1}{2}\|x^0-x^*\|_{\veta}^2-\frac{\beta}{2}\|r^0\|^2\cr
&~+\frac{1}{2n\rho_y}\|y\|^2+\frac{1}{2n\rho_z}\|z\|^2.
\end{align}

Note $\Psi_\beta(x^0,0)=\sum_{j=1}^m [f_j(x^0)]_+^2$. Since $\rho_y,\rho_z \le \frac{\beta}{n}$, plugging \eqref{eq:y-k-bd} and \eqref{eq:z-k-bd} into \eqref{eq:b-1iter-yz-sum} and using the convexity of $f_i$'s yield
\begin{align*}
\EE[\Phi(\bar{x}^{k+1};x^*,y,z)]
\le \frac{1}{1+\frac{k}{n}}\left(C_{x^0}+\frac{1}{2n\rho_y}\|y\|^2+\frac{1}{2n\rho_z}\|z\|^2\right).
\end{align*}
Therefore, we complete the proof by Lemma \ref{lem:pre-rate}.
\end{proof}

\begin{remark}
If $n$ block updates of $x$ costs roughly the same as one full update to $x$, then the results in \eqref{eq:rate-blalm} are comparable to those in \eqref{eq:rate-lalm} by noting their differences in choosing $\rho_y,\rho_z$. One drawback of Theorem \ref{thm:b-rate} is the assumption on the upper bound of $\vell^k$. From \eqref{eq:der-lip-psi}, we see that the upper bound can be pre-calculated if $f_j(x),\,\forall j\in [m]$ are affine. However, in general, it is unknown and dependent on the iterates. Numerically, we can gradually increase $\eta_i$ by a fixed amount or ratio if $\eta_i < \ell_i^k + \beta\|A_i\|^2$ is detected or by backtracking until the following inequality holds:
\begin{equation}\label{eq:choose-eta-ik}
F_{\beta}(x^{k+1},y^k,z^k) \le F_\beta(w^k) + \big\langle \nabla_{x_{ik}} F_\beta(w^k), x^{k+1}_{i_k}-x^k_{i_k}\big\rangle + \frac{\eta_{i_k}}{2}\|x^{k+1}_{i_k}-x^k_{i_k}\|^2. 
\end{equation}
 After finitely many increases, $\ell_i^k + \beta\|A_i\|^2 \le \eta_i,\forall i,$ will hold in high probability for every $k$. This can be explained by the following arguments. 

Let $\eta_i=\zeta\ge 1,\forall i\in [n]$. Since $n\rho_z \le \frac{\beta}{2}$, then from \eqref{eq:opt}, \eqref{eq:b-1iter-yz-sum} with $(y,z)=(y^*,z^*)$, \eqref{eq:y-k-bd}, and \eqref{eq:z-k-bd}, it follows that 
$$\begin{array}{l}\EE\|x^k-x^*\| \le \sqrt{\frac{C_{x^0}}{\zeta}}+\|x^0-x^*\|+\frac{\|y^*\|}{\sqrt{n\rho_y\zeta}}+\frac{\|z^*\|}{\sqrt{n\rho_z\zeta}},\\[0.1cm]\EE\|z^k-z^*\|\le \sqrt{\beta C_{x^0}}+ \sqrt{\beta \zeta}\|x^0-x^*\|+\sqrt{\frac{\beta}{n\rho_y}}\|y^*\|+\sqrt{\frac{\beta}{n\rho_z}}\|z^*\|.\end{array}$$ 
Hence, we have from \eqref{eq:fzj} that $\EE [\ell_i^k] = O(\sqrt{\zeta})$. By the Markov inequality, for every $k$, $\ell_i^k \le \eta_i, \forall i\in [n]$ holds in high probability if $\eta_i\gg \underline{\zeta},\forall i\in[n]$, where $\underline{\zeta}\ge 1$ satisfies
\begin{align*}
\underline{\zeta} \ge &~ L_g +\beta\|A\|^2 + \beta\sum_{j=1}^m B_j^2 + \beta \sum_{j=1}^m B_j L_j \left(\sqrt{C_{x^0}}+\|x^0-x^*\|+\frac{\|y^*\|}{\sqrt{n\rho_y}}+\frac{\|z^*\|}{\sqrt{n\rho_z}}\right)\\
 &~+ \sqrt{\sum_{j=1}^m L_j^2} \left(\|z^*\|+\sqrt{2\beta C_{x^0}}+ \sqrt{2\beta \underline{\zeta}}\|x^0-x^*\|+\sqrt{\frac{2\beta}{n\rho_y}}\|y^*\|+\sqrt{\frac{2\beta}{n\rho_z}}\|z^*\|\right).
\end{align*}  
\end{remark}

\begin{remark}
If Assumption \ref{assump-scp} is satisfied, we can also show a local linear convergence result of Algorithm \ref{alg:blalm} following the analysis in section \ref{sec:local} and that in \cite{xu2017accelerated-pdbc}. We do not expand details here but leave it to interested readers.
\end{remark}

\section{Applications}\label{sec:application}
In this section, we give a few applications that can be formulated in the form of \eqref{eq:ccp} and discuss how Algorithm \ref{alg:lalm} and/or Algorithm \ref{alg:blalm} can be applied.
\subsection{Basis pursuit denosing}
Suppose we observe a noisy measurement $b=A\theta^o+\xi$ of a signal $\theta^o$, where $A$ is a measuring matrix, and $\xi$ is a noise vector. Assume $\theta^o$ can be sparsely represented by a dictionary $D$. Then we can recover the signal through solving the so-called basis pursuit denoising (BPDN) problem:
\begin{equation}\label{eq:bpdn}
\min_x \|x\|_1, \st \|ADx-b\|^2 \le \delta,
\end{equation}
where $\delta$ measures the noise level. Upon obtaining a solution $x^*$ to \eqref{eq:bpdn}, we let $\theta^r = Dx^*$ be the recovered signal. Depending on the application, one can impose certain bounds on $x$ to make the recovered signal physically meaningful. In this case, all conditions in Assumption \ref{assump-lip} holds. In addition, assuming $b\in \mathrm{Range}(AD)$, then Slater condition holds, and thus Assumption \ref{assump-kkt} is satisfied. Hence, Algorithm \ref{alg:lalm} is applicable, and the $x$-subproblem \eqref{eq:lalm-x} has closed-form solution by shrinkage or soft-thresholding.
If $A$ and $D$ are stored as matrices, \eqref{eq:bpdn} is coordinate friendly, and we can also use Algorithm \ref{alg:blalm}. However, for certain signal processing problems, evaluating $A\theta$ and/or $Dx$ may not require explicit form of $A$ or $D$ but can be efficiently realized, such as a partial circulant $A$ and/or a discrete cosine dictionary $D$. For this case, Algorithm \ref{alg:blalm} will not be as efficient as Algorithm \ref{alg:lalm} since evaluating coordinate gradient of $\|ADx-b\|^2$ may require full gradient. 

\subsection{Quadratically constrained quadratic programming}
The quadratically constrained quadratic programming (QCQP) can be formulated as
\begin{equation}\label{eq:qcqp}
\begin{aligned}
\min_{x\in \RR^p}~ & \frac{1}{2}x^\top Q_0 x + c_0^\top x + d_0\\
\st & \frac{1}{2}x^\top Q_j x + c_j^\top x + d_j \le 0,\, \forall j \in [m],\\
& Ax = b,\\
& l_i \le x_i \le u_i, \forall i\in  [p].
\end{aligned}
\end{equation}
Let $\cX=[l_1,u_1]\times\cdots\times [l_p,u_p]$ and $h(x) = \iota_\cX(x)$. Then \eqref{eq:qcqp} can be written as \eqref{eq:ccp} by adding $h(x)$ into the objective. When every $Q_j$ is positive semidefinite, the problem is convex, and if all $l_i$'s and $u_i$'s are finite, then $\cX$ is bounded and all conditions in Assumption \ref{assump-lip} hold. Hence, we can apply Algorithm \ref{alg:lalm} to find a solution of \eqref{eq:qcqp}, and the solution of $x$-subproblem \eqref{eq:lalm-x} can be explicitly given by performing projection to a box constraint. In addition, the problem is coordinate friendly since evaluating the partial derivative of $\frac{1}{2}x^\top Q_j x + c_j^\top x + d_j$ about each $x_i$ costs roughly $\frac{1}{p}$ of computing the full gradient. Furthermore, if we maintain $Q_j x$, then calculating the function value is negligible compared to the gradient computation. Therefore, we can also apply Algorithm \ref{alg:blalm} to the QCQP.


\subsection{Finite minimax problems}
Many applications can be formulated as a finite minimax problem (e.g., see \cite{pee2011solving} and the references therein):
\begin{equation}\label{eq:minimax}
\min_{x\in \cX} \max_{1\le j\le m} f_j(x),
\end{equation}
where each $f_j$ is a smooth convex function. Although all $f_j$'s are differentiable, the objective of \eqref{eq:minimax} is generally not differentiable due to the max operation. Introducing variable $t$ and requiring $\max_{1\le j\le m} f_j(x)\le t$, one can express the minimax problem equivalently to
\begin{equation}\label{eq:minimax-ccp}
\min_{x\in\cX, t} t, \st f_j(x)-t \le 0, \,\forall j\in [m].
\end{equation}
For any $x\in\mathrm{int}(\cX)$, each inequality constraint holds strictly at $(x, \max_j f_j(x) + 1)$, and thus the Slater condition holds. Hence, Assumption \ref{assump-kkt} is satisfied. In addition, if $\cX$ is bounded, then all conditions in Assumption \ref{assump-lip} also hold. Therefore, we can use Algorithm \ref{alg:lalm} to find a solution of \eqref{eq:minimax-ccp} and equivalently \eqref{eq:minimax}, and every iteration requires performing a projection to $\cX$. Depending on applications, one may also apply Algorithm \ref{alg:blalm} if the problem is coordinate friendly, for example, every $f_j$ is a quadratic function.

\section{Numerical experiments}\label{sec:numerical}
In this section, we test Algorithms \ref{alg:lalm} and \ref{alg:blalm} on BPDN \eqref{eq:bpdn} and QCQP \eqref{eq:qcqp} to show their numerical performance. The two algorithms are named as LALM and BLALM respectively. For both algorithms, we choose the parameter $\eta$ by backtracking. More precisely, at each iteration $k$, for LALM, we start from $\eta^k=\eta^{k-1}$ and multiply it by 1.5 if \eqref{eq:choose-eta} fails, and for BLALM, we initialize $\eta_{i_k}^k=\eta_{i_k}^{k-1}$ and multiply it by 1.5 if \eqref{eq:choose-eta-ik} does not hold. For both tests, we run the compared methods to $10^5$ epochs, where one epoch is equivalent to $n$ block updates. Optimal solutions to both tested problems are computed by CVX \cite{grant2008cvx} with high precision.

\subsection{Basis pursuit denoising}\label{sec:bpdn}
In this test, we show the convergence speed of LALM and BLALM on solving BPDN \eqref{eq:bpdn}. For simplicity, we set $D=I$. The matrix $A\in \RR^{50\times 100}$ is randomly generated according to the standard Gaussian distribution, and the underlying sparse signal $x^o$ has 5 nonzero components following the standard Gaussian distribution. Then we let $b=Ax^0 + 0.1\xi$, where $\xi$ is a unit Gaussian noise vector. For BLALM, we evenly partition the variable $x$ into 10 blocks.  The parameter $\beta$ is simply set to 1 for both methods, and $\rho_z=\beta$ is set for LALM and $\rho_z=\frac{\beta}{10}$ for BLALM. Note that for the latter, the value of $\rho_z$ is larger than that given by the theorem, and the algorithm still works well. This may indicate that our analysis is not tight. 

\begin{figure}
\begin{center}
\includegraphics[width=0.45\textwidth]{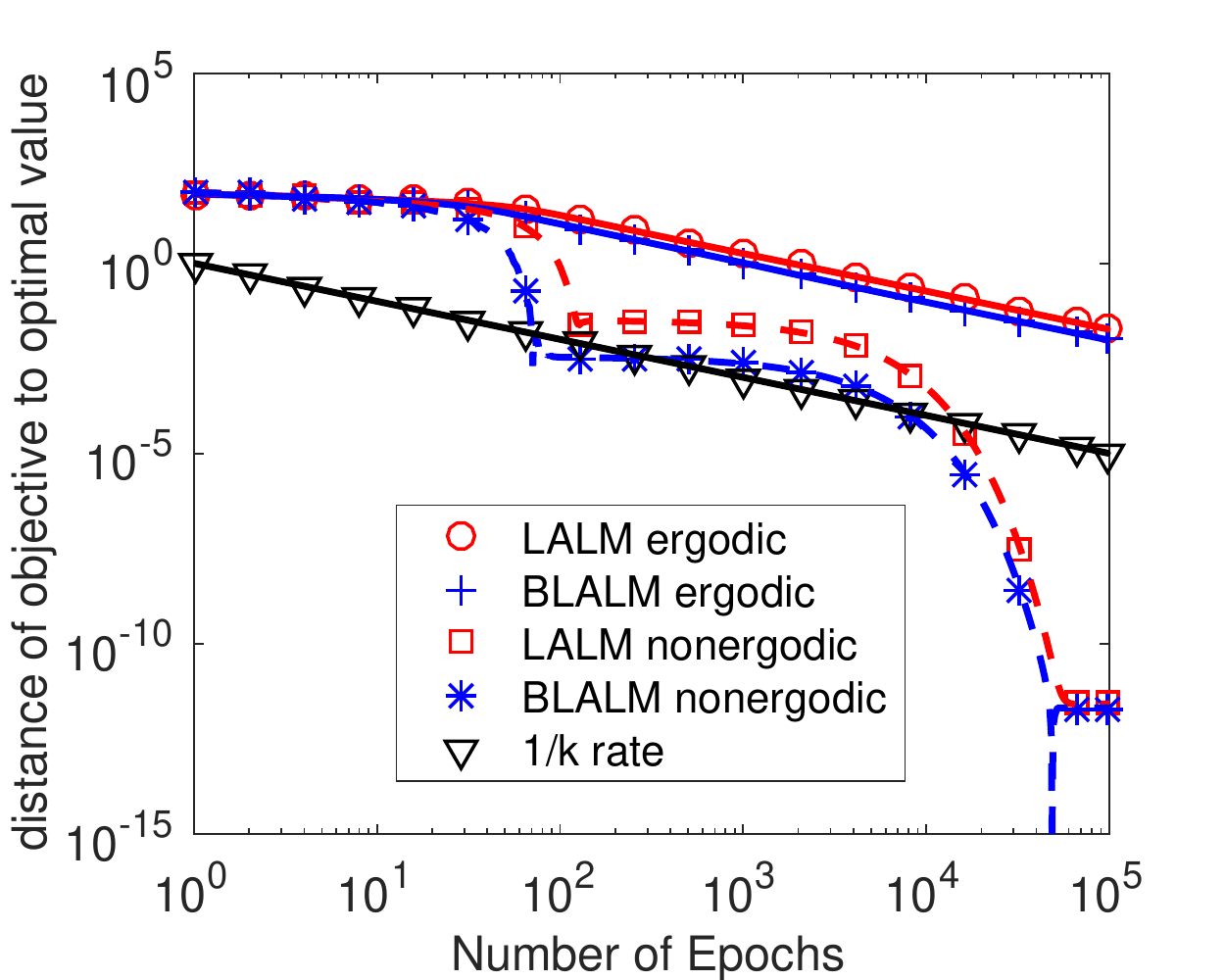}
\includegraphics[width=0.45\textwidth]{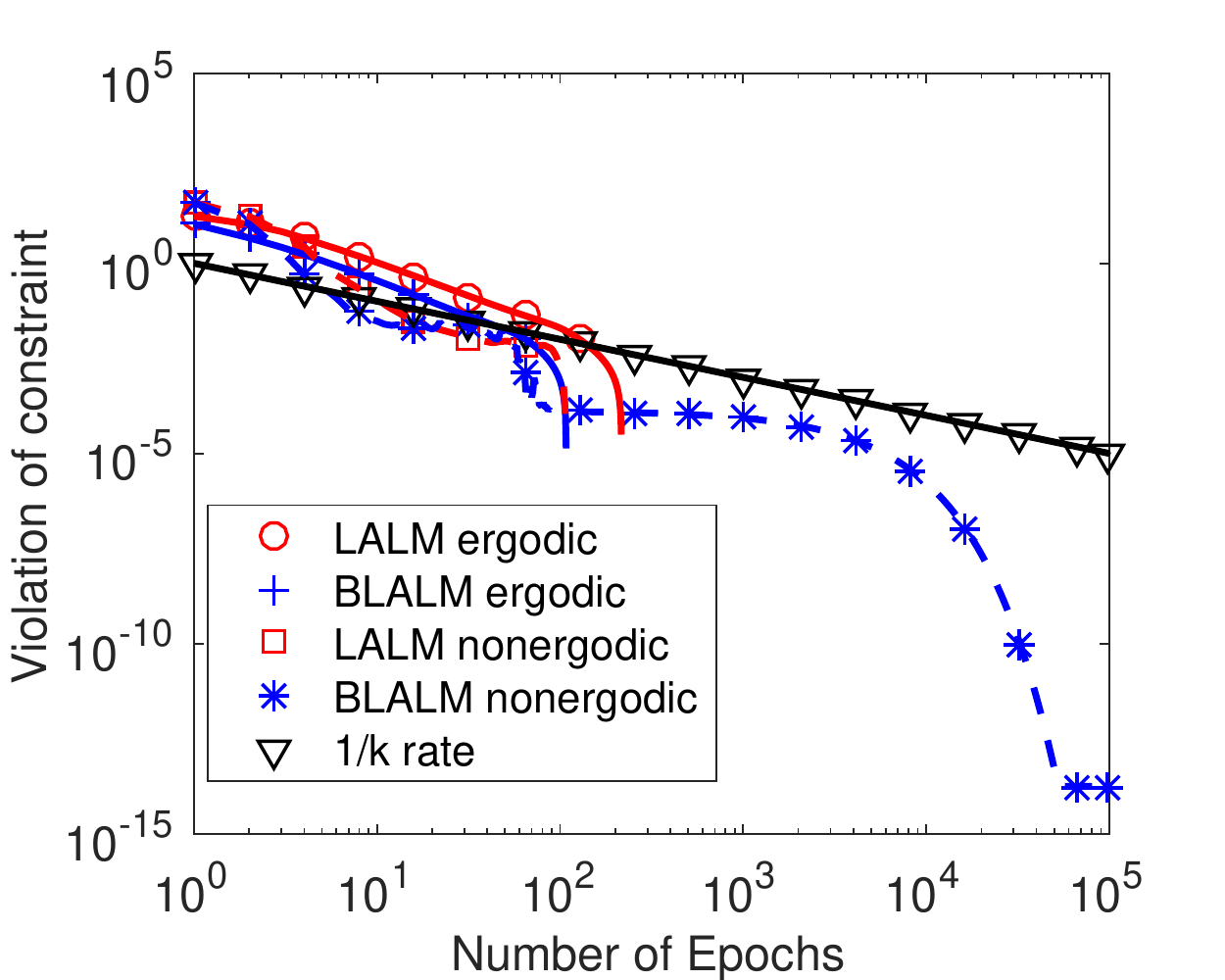}
\end{center}
\caption{Convergence behaviors of Algorithm \ref{alg:lalm} (named LALM) and Algorithm \ref{alg:blalm} (named BLALM) on the BPDN problem \eqref{eq:bpdn}. Left: distance of objective value to the optimal value $|f_0(x)-f_0(x^*)|$; Right: constraint residual $\sum_{j=1}^m[f_j(x)]_+$. ``ergodic'' curves are measured by averaged iterate $\bar{x}^k$ and ``nonergodic'' ones by actual iterate $x^k$. The missing part on each constraint violation curve corresponds to zero residual.}\label{fig:bpdn}
\end{figure}

Figure \ref{fig:bpdn} plots the objective values and constraint residuals produced by both algorithms, where the curve corresponding to ``ergodic'' is obtained by using the averaged iterates $\bar{x}^k$ and ``nonergodic'' by the actual iterate $x^k$. The missing part on each constraint violation curve corresponds to zero residual. Since LALM and BLALM have similar per-epoch complexity, their comparison in terms of running time is similar to that in Figure \ref{fig:bpdn}. From the figure, we see that BLALM is better than LALM in terms of both ergodic and nonergodic iterates. The ergodic convergence speed of both methods is precisely the order of $\frac{1}{k}$ and matches our theorems. However, the nonergodic convergence is significantly faster, especially as the iterate approaches to optimality. This is possibly because the iterate enters a region where the algorithm has linear convergence as indicated by the analysis in section \ref{sec:local}. For this reason, we use the actual iterate in the next test.

\subsection{Quadratically constrained quadratic programming}\label{sec:qcqp}
In this subsection, we test LALM and BLALM on the QCQP problem \eqref{eq:qcqp} and compare them to the recently proposed first-order primal-dual type method by Yu\&Neely \cite{yu2016primal}. They consider a smooth constrained convex program in the form of \eqref{eq:ccp} without an explicit linear equality constraint. Their method that we name as PD-YN iteratively performs the updates:
\begin{subequations}\label{eq:pd-yn}
\begin{align}
&x^{k+1}=\cP_\cX\left(x^k-\frac{1}{\eta}\left[\nabla f_0(x^k)+\sum_{j=1}^m\big(\lambda_j^k+f_j(x^k)\big)\nabla f_j(x^k)\right]\right),\\
&\lambda_j^{k+1}=\max\big(-f_j(x^k),\lambda_j^k+f_j(x^k)\big),\,\forall j\in [m],
\end{align}
\end{subequations}
where $\lambda_j^0 = \max(0, -f_j(x^0)),\,\forall j\in [m]$, and $\eta$ is the step size. In the test, we also choose $\eta$ adaptively by backtracking such that
$$\phi(x^{k+1},z^k) \le \phi(x^k,z^k) + \big\langle \nabla_x\phi(x^k,z^k), x^{k+1}-x^k\big\rangle + \frac{\eta^k}{2}\|x^{k+1}-x^k\|^2,$$
where $\phi(x,z) = f_0(x) + \sum_{j=1}^m z_j f_j(x)$ and $z^k_j=\lambda_j^k+f_j(x^k)$. Although \cite{yu2016primal} does not show the convergence of PD-YN with the above adaptive $\eta^k$, we observe its better performance than that with a fixed $\eta$. 

The problem size is set to $m=10$ and $p=2000$. We randomly generate SPD matrices $Q_j, j=0,1,\ldots m$. $A$ is set to zero, i.e., there is no linear equality constraint. The vector $c_j$'s are generated according to the standard Gaussian distribution, and $d_0=0$ and each $d_j$ is a negative number for $j\in [m]$. Also we set $l_i=-10$ and $u_i=10$ for each $i\in[p]$. Hence, the zero vector is an interior point of $\cX$ and makes every inequality hold strictly, namely, the Slater condition holds. We set $\beta=0.1$ for both LALM and BLALM, and for the latter, we evenly partition the variable into 200 blocks. The parameter $\rho_z$ is set to $\beta$ and $\frac{\beta}{200}$ for the two algorithms respectively.

\begin{figure}
\begin{center}
\includegraphics[width=0.45\textwidth]{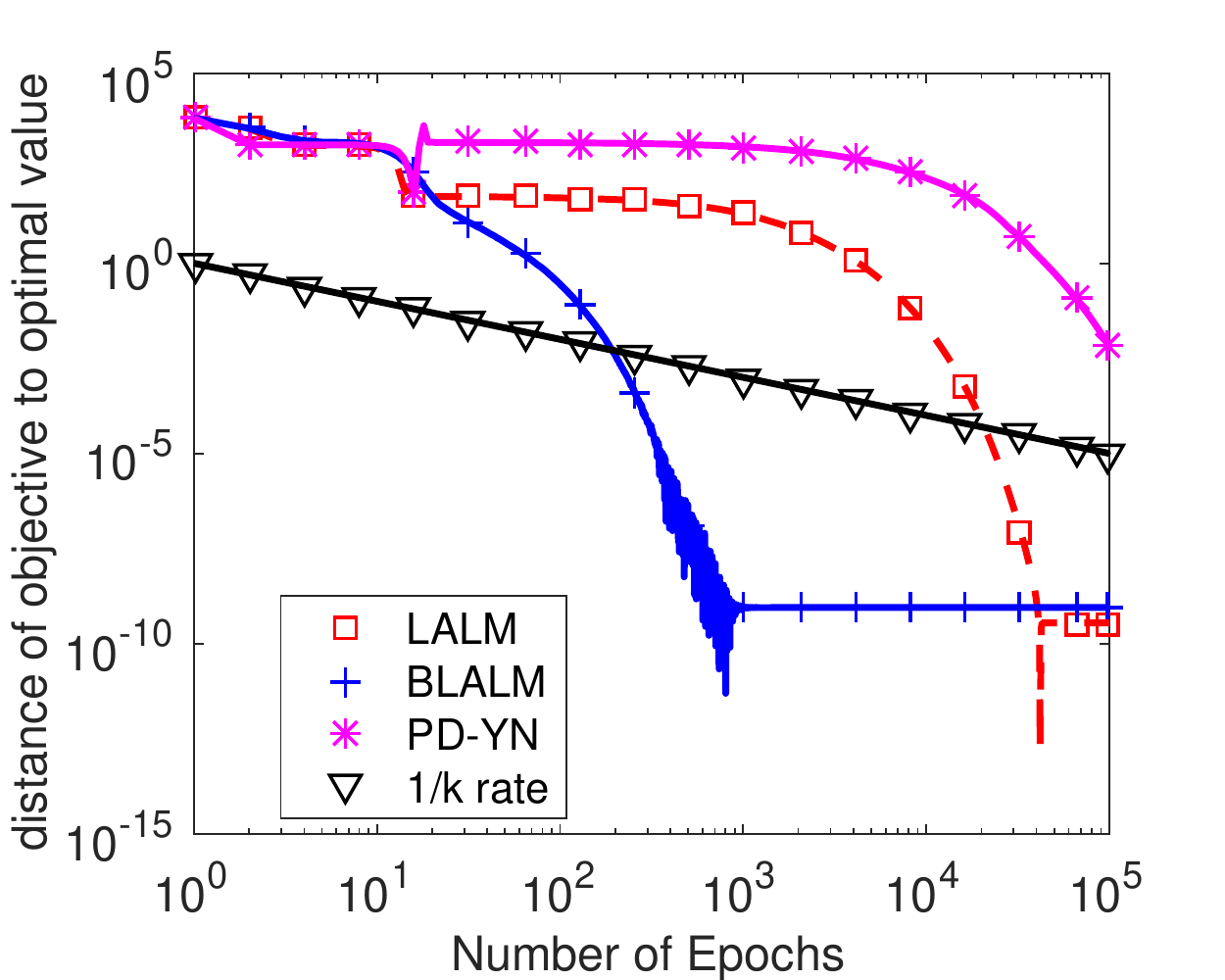}
\includegraphics[width=0.45\textwidth]{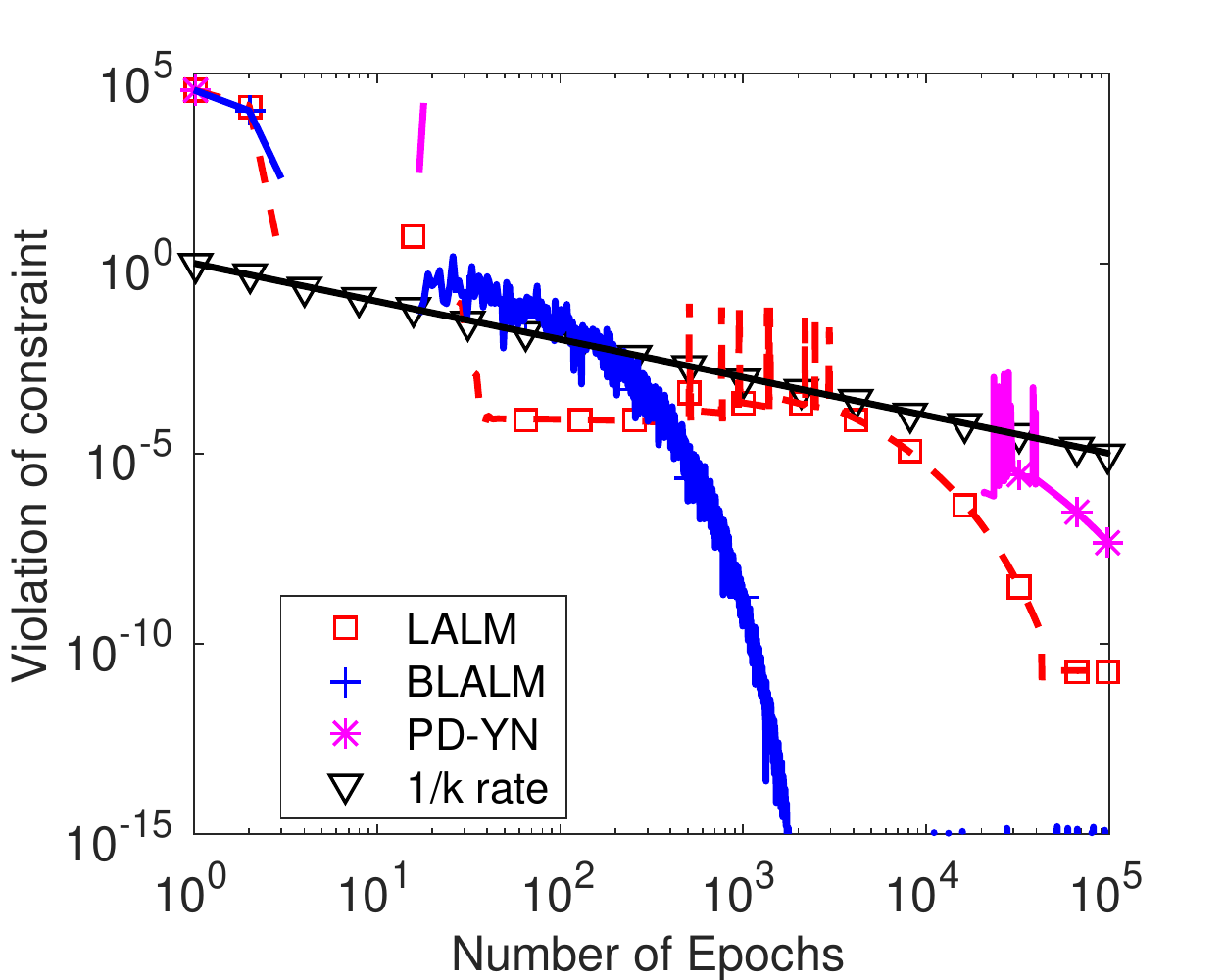}
\end{center}
\caption{Convergence behaviors of Algorithm \ref{alg:lalm} (named LALM), Algorithm \ref{alg:blalm} (named BLALM), and the primal-dual type method (named PD-YN) in \cite{yu2016primal} on the QCQP problem \eqref{eq:qcqp}. Left: distance of objective value to the optimal value $|f_0(x)-f_0(x^*)|$; Right: constraint residual $\sum_{j=1}^m[f_j(x)]_+$. The missing part on each constraint violation curve corresponds to zero residual.}\label{fig:qcqp}
\end{figure}

Figure \ref{fig:qcqp} plots the results by the three compared algorithms. Both the proposed methods perform significantly better than PD-YN, and BLALM is the best among the three. In addition, we notice that LALM and BLALM converge linearly when the iterate approaches to optimality, as indicated by Theorem \ref{thm:loc-lin}.

\section{Conclusions}\label{sec:conclusion}
We have presented a first-order method for solving composite convex programming with both equality and smooth nonlinear inequality constraints. The method is derived from proximal linearization of the classic augmented Lagrangian function. Its global iterate convergence and global sublinear and local linear convergence results have been established. For the problem that has coordinate friendly structure, we have also proposed a first-order randomized block update method and shown its global sublinear convergence in expectation. In addition, we have implemented the two methods on solving the basis pursuit denoising problem and the convex quadratically constrained quadratic programming. Global sublinear and local linear convergence are both observed in the numerical experiments.

\bibliographystyle{abbrv}
\bibliography{optim}

\end{document}

%% file: macros.tex
\usepackage{mathtools}

\usepackage[normalem]{ulem} 

\newcommand{\bm}[1]{\boldsymbol{#1}}

\newcommand{\veta}{{\bm{\eta}}}
\newcommand{\vell}{{\bm{\ell}}}

\newcommand{\cB}{{\mathcal{B}}}

\newcommand{\cK}{{\mathcal{K}}}
\newcommand{\cL}{{\mathcal{L}}}

\newcommand{\cP}{{\mathcal{P}}}

\newcommand{\cW}{{\mathcal{W}}}
\newcommand{\cX}{{\mathcal{X}}}

\newcommand{\vareps}{\varepsilon}

\newcommand{\EE}{\mathbb{E}} 
\newcommand{\RR}{\mathbb{R}} 

\newcommand{\dom}{{\mathrm{dom}}} 

\newcommand{\Null}{{\mathrm{Null}}} 

\newcommand{\st}{\mbox{ s.t. }}

\DeclareMathOperator*{\argmin}{arg\,min} 

\newcommand{\bc}{\begin{center}}
\newcommand{\ec}{\end{center}}

\newcommand{\bdm}{\begin{displaymath}}
\newcommand{\edm}{\end{displaymath}}

\newcommand{\beq}{\begin{equation}}
\newcommand{\eeq}{\end{equation}}

\newcommand{\bfl}{\begin{flushleft}}
\newcommand{\efl}{\end{flushleft}}

\newcommand{\bt}{\begin{tabbing}}
\newcommand{\et}{\end{tabbing}}

\newcommand{\beqn}{\begin{eqnarray}}
\newcommand{\eeqn}{\end{eqnarray}}

\newcommand{\beqs}{\begin{align*}} 
\newcommand{\eeqs}{\end{align*}}  

\newtheorem{theorem}{Theorem}[section]
\newtheorem{definition}{Definition}[section]

\newtheorem{lemma}{Lemma}[section]
\newtheorem{proposition}{Proposition}[section]
\newtheorem{remark}{Remark}[section]
\newtheorem{assumption}{Assumption}